\documentclass[12pt]{amsart}
\usepackage{amsmath,amsfonts,euscript,amscd,amsthm,amssymb,upref,graphics}
\usepackage{mathrsfs}
\usepackage[all]{xy}
\usepackage{color}

\theoremstyle{definition}

\swapnumbers

\theoremstyle{plain}

\newtheorem{theorem}{Theorem}[section]

\newtheorem{proposition}[theorem]{Proposition}
\newtheorem{lemma}[theorem]{Lemma}

\newtheorem{corollary}[theorem]{Corollary}

\newtheorem{sub}{}[theorem] % This creates the counter "sub"

\newtheorem{subproposition}	[sub]{Proposition}
\newtheorem{sublemma}	[sub]{Lemma}

\theoremstyle{definition}

\newtheorem{subdefinition}[sub]{Definition}

\newtheorem{definition}[theorem]{Definition}

\newtheorem{parag}[theorem]{}
\newtheorem{example}[theorem]{Example}

\newtheorem{notation}[theorem]{Notation}

\newtheorem{subparag}[sub]{}

\newtheorem{bigremark}[theorem]{Remark}
\newtheorem{subbigremark}[sub]{Remark}

\theoremstyle{remark}

\newtheorem*{remark}{Remark}
\newtheorem*{remarks}{Remarks}

{\begin{enumerate}\setlength{\itemsep}{#1}}{\end{enumerate}}
{\vspace{2mm}\noindent\ref{#1}. {\bf #2.} \it}{\vspace{2mm}}

\newcommand{\Aut}{	\operatorname{{\rm Aut}}}
\newcommand{\Spec}{	\operatorname{\text{\rm Spec}}}

\newcommand{\haut}{	\operatorname{\text{\rm ht}}}

\newcommand{\image}{	\operatorname{\text{\rm im}}}
\newcommand{\trdeg}{	\operatorname{\text{\rm trdeg}}}
\newcommand{\Frac}{	\operatorname{\text{\rm Frac}}}
\newcommand{\Char}{	\operatorname{\text{\rm char}}}
\newcommand{\Span}{	\operatorname{\text{\rm Span}}}
\newcommand{\id}{	\operatorname{\text{\rm id}}}
\newcommand{\Der}{	\operatorname{\text{\rm Der}}}

\newcommand{\Nil}{	\operatorname{\text{\rm Nil}}}
\newcommand{\lnd}{	\operatorname{\text{\rm LND}}}
\newcommand{\klnd}{	\operatorname{\text{\rm KLND}}}
\newcommand{\ML}{	\operatorname{\text{\rm ML}}}
\newcommand{\FML}{	\operatorname{\text{\rm FML}}}
\newcommand{\lndrank}{	\operatorname{{\rm lndrk}}}
\newcommand{\Gal}{	\operatorname{\text{\rm Gal}}}

\newcommand{\Dan}{	\operatorname{{\Dgoth}}}
\newcommand{\pl}{	\operatorname{{\rm pl}}}

\newcommand{\Ceuls}{\Ceul_{\text{\rm s}}}

\newcommand{\setspec}[2]{\big\{\,#1\, \mid \,#2\, \big\}}

\newcommand{\Integ}{\ensuremath{\mathbb{Z}}}
\newcommand{\Nat}{\ensuremath{\mathbb{N}}}
\newcommand{\Rat}{\ensuremath{\mathbb{Q}}}
\newcommand{\Comp}{\ensuremath{\mathbb{C}}}
\newcommand{\Reals}{\ensuremath{\mathbb{R}}}

\newcommand{\proj}{\ensuremath{\mathbb{P}}}
\newcommand{\bk}{{\ensuremath{\rm \bf k}}}
\newcommand{\ck}{{\ensuremath{\bar\bk}}}
\newcommand{\kk}[1]{\bk^{[#1]}}

\newcommand{\Dgoth}{{\ensuremath{\mathfrak{D}}}}

\newcommand{\pgoth}{{\ensuremath{\mathfrak{p}}}}
\newcommand{\qgoth}{{\ensuremath{\mathfrak{q}}}}

\newcommand{\mgoth}{{\ensuremath{\mathfrak{m}}}}

\newcommand{\Beul}{\EuScript{B}}
\newcommand{\Ceul}{\EuScript{C}}

\newcommand{\Eeul}{\EuScript{E}}

\newcommand{\Neul}{\EuScript{N}}

\newcommand{\Xeul}{\EuScript{X}}

\newcommand{\Ascr}{\mathscr{A}}
\newcommand{\Kscr}{\mathscr{K}}

\renewcommand{\epsilon}{\varepsilon}
\renewcommand{\phi}{\varphi}
\renewcommand{\emptyset}{\varnothing}

\newenvironment{enumerata}%
{\begin{enumerate}

}{\end{enumerate}}

\newcommand{\rien}[1]{}

\begin{document}
\renewcommand{\baselinestretch}{1.07}

\title[Locally nilpotent derivations and the structure of rings]{Locally nilpotent derivations\\ and the structure of rings}

\author{Daniel Daigle}

\address{Department of Mathematics and Statistics\\
	University of Ottawa\\
	Ottawa, Canada\ \ K1N 6N5}

\email{ddaigle@uottawa.ca}

\thanks{Research supported by a grant from NSERC Canada.}

\keywords{Locally nilpotent derivation, unirational variety, rational variety, affine variety, affine space, polynomial ring.}

{\renewcommand{\thefootnote}{}
\footnotetext{2010 \textit{Mathematics Subject Classification.}
Primary: 14R10. Secondary: 14R20, 14M20, 14R05.}}

\begin{abstract} 
We investigate the structure of commutative integral domains $B$ of characteristic zero by studying the kernels
of locally nilpotent derivations $D : B \to B$.
\end{abstract}
\maketitle
  
\vfuzz=2pt

%%%%%%%%%%%%%%%%%%%%%%%%%%%%%%%%%%%%%%%%%%%%%%%%%%%%%%%%%%%%%%%%%%%
%%%%%%%%%%%%%%%%%%%%%%%%%%%%%%%%%%%%%%%%%%%%%%%%%%%%%%%%%%%%%%%%%%%
%%%%%%%%%%%%%%%%%%%%%%%%%%%%%%%%%%%%%%%%%%%%%%%%%%%%%%%%%%%%%%%%%%%

\section*{Introduction}

In this article, \textit{ring} means commutative ring with a unity
and \textit{domain} means commutative integral domain.
If $A$ is a domain then $\Frac(A)$ is its field of fractions.
If $\bk$ is a field then a \textit{$\bk$-domain} is a domain that is also a $\bk$-algebra,
and an \textit{affine $\bk$-domain} is a $\bk$-domain that is finitely generated as a $\bk$-algebra.
If $B$ is a ring, a derivation $D: B \to B$ is \textit{locally nilpotent} if for each $x\in B$ there exists $n \in \Nat$ such that $D^n(x)=0$.
The set of locally nilpotent derivations $D : B \to B$ is denoted $\lnd(B)$.
The Makar-Limanov invariant $\ML(B)$ of $B$ and the ``field'' Makar-Limanov invariant $\FML(B)$ of $B$  are defined by:
$$
\ML(B) = \bigcap_{\mbox{\scriptsize $D \in \lnd(B)$}} \ker D 
\qquad \text{and} \qquad
\FML(B) = \bigcap_{\mbox{\scriptsize $D \in \lnd(B)$}} \Frac(\ker D) ,
$$
where in the second case $B$ is assumed to be a domain and the intersection is taken in $\Frac B$.
If $\bk$ is a field of characteristic zero and $B$ is a $\bk$-domain then $\bk \subseteq \ML(B) \subseteq \FML(B)$.

\medskip
When locally nilpotent derivations are used for studying the structure of a domain $B$ of characteristic zero,
one typically pays attention to $\ML(B)$, or to $\FML(B)$, or to an individual $\ker(D)$ for some $D \in \lnd(B)$.
The present article proposes to refine this idea and to consider, in a systematic way, all  rings and fields 
$A_{\Delta} = \bigcap_{D \in \Delta} \ker D$ and $K_{\Delta} = \bigcap_{D \in \Delta} \Frac(\ker D)$, where $\Delta$ can be any subset of $\lnd(B)$.
The lattices $\Ascr(B) = \setspec{ A_\Delta }{ \Delta \subseteq \lnd(B) }$ and $\Kscr(B) = \setspec{ K_\Delta }{ \Delta \subseteq \lnd(B) }$
are defined at the beginning of Section \ref{Section:whatisneededfornextsection} and are used throughout the article.

One of the guiding principles in this area is the idea that if an affine $\bk$-domain $B$ admits many locally nilpotent derivations
then $B$ should be close to being rational over $\bk$.
A brief review of the history of this idea is given in Section \ref{SectionSomeClarifications}, 
together with a clarification of some issues related to the problematic status of certain claims that have been published with invalid proofs.
Some of the results given in Section \ref{SectionSomeClarifications} are stronger than the statements that are being repaired or revisited,
and others are altogether new.
The results numbered \ref{cxvoq9wedowjskdpx}, \ref{ckjv9hrjdulxjkdxlkbrvn}, \ref{0xc9vj2o3w9edno9} and \ref{cknvlwr90fweodpx} are particularly interesting.

Sections \ref{Section:whatisneededfornextsection} and \ref{Somepreliminaries} are mostly devoted to establishing the properties
of $\Ascr(B)$ and $\Kscr(B)$ needed in the rest of the article.
In fact Section \ref{Section:whatisneededfornextsection} presents only the small amount of theory that is needed in Section \ref{SectionSomeClarifications}
($\Ascr(B)$ and $\Kscr(B)$ play only a minor role in  Section \ref{SectionSomeClarifications})
and Section \ref{Somepreliminaries} contains a more extensive study of the two lattices.

Section \ref{Section:applicationsandquestions} applies the ideas developed in the previous sections
to affine $\bk$-domains $B$, where $\bk$ is any field of characteristic zero (most results are still interesting when $\bk$ is assumed to
be algebraically closed).
Having new invariants of rings allows the formulation of new questions, some of which can be answered. 
Thm \ref{doccvobjmExm48wednce} is a general result (for $B$ normal) about 
chains $A_0 \subset A_1 \subset \cdots \subset A_n$ of elements of $\Ascr(B)$ satisfying $\trdeg(A_i : A_{i-1})=1$ for all $i$;
the case where $B$ is a UFD (Cor.\ \ref{ld0b23or0f9vwo4398hbg}) has a particularly pleasant statement.
Under certain assumptions regarding factoriality and units of $B$,
Thm \ref{Z0cjh93wijbrfpa98gfIa} gives information about the elements $A$ of $\Ascr(B)$ satisfying $\trdeg_\bk(A) \le 2$.
Thm \ref{d0qk3bfvo09e8u37} answers the following natural question.
For simplicity, assume that $\bk$ is algebraically closed.
It is known that the condition $\FML(B)=\bk$ implies that $B$ is unirational---but not necessarily rational---over $\bk$.
{\it Can one formulate a condition on the locally nilpotent derivations of $B$ that would imply rationality?}
Thm \ref{d0qk3bfvo09e8u37} implies (in particular) that
if there exists a chain $A_0 \subset A_1 \subset \cdots \subset A_n$ of elements of $\Ascr(B)$ satisfying $n=\dim B$,
then $B$ is rational over $\bk$.

Section \ref{Section:Formsofkknandofkk3} explains how the general results of Section \ref{Section:applicationsandquestions}
apply to certain special classes of algebras that include in particular all forms of $\kk n$,
all stably polynomial algebras over $\bk$, and all exotic $\Comp^n$.

% This article lays out a method and gives a sample of applications.

\subsection*{Notations}
To the notations and conventions already introduced in the above text, we add the following.
We write $\subseteq$ for inclusion, $\subset$ for strict inclusion, $\setminus$ for set difference,
and we agree that $0 \in \Nat$.
If $A$ is a ring and $n \in \Nat$, $A^{[n]}$ denotes a polynomial ring in $n$ variables over $A$;
if $\bk$ is a field, $\bk^{(n)}$ denotes the field of fractions of $\kk n$.
We write $\trdeg_K(L)$ or $\trdeg(L:K)$ for the transcendence degree of a field extension $L/K$.
If $A \subseteq B$ are domains, the transcendence degree of $B$ over $A$ is defined to be that of $\Frac B$ over $\Frac A$,
and is denoted $\trdeg_A(B)$ or $\trdeg(B:A)$.
If $A$ is a ring then $A^*$ is its group of units, $\dim A$ is the Krull dimension of $A$,
and if $a \in A$ then $A_a = S^{-1}A$ where $S = \{1,a,a^2,\dots\}$.
If $R \subseteq A$ are domains then $A_R = S^{-1}A$ where $S = R \setminus \{0\}$
(note that $A_R$ is an algebra over the field $R_R = \Frac(R)$).

\section{Preliminaries}
\label {Section:whatisneededfornextsection}

This Section presents the material that is needed in Section \ref{SectionSomeClarifications}.
Gathering this material here will enable us to go through Section \ref{SectionSomeClarifications} without interrupting the flow of the discussion.
We begin by recalling some basic facts about locally nilpotent derivations.
For background on this topic, we refer the reader to any of \cite{VDE:book}, \cite{Freud:Book-new} or \cite{Dai:IntroLNDs2010}.

\begin{parag} \label {pc9293ed0wdjo03}
Let $B$ be a domain of characteristic zero. Let $D \in \lnd(B) \setminus \{0\}$ and $A = \ker D$.
\begin{enumerate}

\item[(i)] The ring $A$ is factorially closed in $B$, i.e.,
the implication $xy \in A \Rightarrow x,y \in A$ is true for all $x,y \in B \setminus \{0\}$.
It follows that $A^* = B^*$ and hence that if $\bk$ is any field included in $B$ then $\bk\subseteq A$.
Moreover, if $B$ is a UFD then so is $A$.

\item[(ii)] The  \textbf{Slice Theorem} asserts that  if $\Rat \subseteq B$ and $s \in B$ is such that $D(s) \in B^*$ then $B = A[s] = A^{[1]}$.
(Refer to \cite[Prop.\ 2.1]{Wright:JacConj} for this result.)

\item[(iii)] Clearly, there exists $s \in B$ satisfying $D(s) \neq 0$ and $D^2(s)=0$. If $\Rat \subseteq B$ then
for any such $s$ we have $B_a = A_a[s] = A_a^{[1]}$, where we set $a = D(s)$.

\item[(iv)] If we define $K = \Frac A$ then $B_A = K^{[1]}$ and $\Frac B = K^{(1)}$; in particular, $K$ is algebraically closed in $\Frac B$.
(The notation $B_A$ is defined at the end of the introduction.)

\item[(v)] If $f \in B$ satisfies $f \mid D(f)$, then $D(f)=0$.

\end{enumerate}
\end{parag}

The following concepts play a major role in this article.

\begin{definition} \label {0cd20FXGICHGgt86cjF5i5rgo909}
Let $B$ be a domain of characteristic zero.
Given a subset $\Delta$ of $\lnd(B)$, define
$A_{\Delta} = \bigcap_{D \in \Delta} \ker D$ and 
$K_{\Delta} = \bigcap_{D \in \Delta} \Frac(\ker D)$,
where the first intersection is taken in $B$ and the second in $\Frac B$
(in particular,  $A_\emptyset = B$ and $K_\emptyset = \Frac B$).
Then define the two sets
$$
\Ascr(B) = \setspec { A_{\Delta} }{ \Delta \subseteq \lnd(B)  } 
\quad \text{and} \quad
\Kscr(B) = \setspec { K_{\Delta} }{ \Delta \subseteq \lnd(B)  } .
$$
We view these as posets: $(\Ascr(B), \subseteq)$ and $(\Kscr(B), \subseteq)$.
Note that $\Ascr(B)$ is a nonempty set of subrings of $B$; its greatest element is $B$ and its least element is $\ML(B)$.
Similarly, $\Kscr(B)$ is a nonempty set of subfields of $\Frac B$ whose greatest element is $\Frac B$ and whose least element is $\FML(B)$.
For each $n \in \Nat$, define
$$
\Ascr_n(B) = \setspec{ A \in \Ascr(B) }{ \trdeg_A( B) = n } \text{\ \ and\ \ } \Kscr_n(B) = \setspec{ K \in \Kscr(B) }{ \trdeg_K( \Frac B) = n } .
$$
Observe that $\Ascr_0(B) = \{B\}$, $\Kscr_0(B) = \{ \Frac B \}$,
$$
\Ascr_1(B) = \setspec{ \ker D }{ D \in \lnd(B)\setminus\{0\} } \text{\ \ and\ \ } \Kscr_1(B) = \setspec{ \Frac(\ker D) }{ D \in \lnd(B)\setminus\{0\} } .
$$
The set $\Ascr_1(B)$ is sometimes denoted $\klnd(B)$ (but not in this article).
Also keep in mind that, when $n>1$, the elements of $\Kscr_n(B)$ are not necessarily the fields of fractions of those of $\Ascr_n(B)$.
\end{definition}

\begin{definition}  
Let $B$ be a domain of caracteristic zero.
One says that $B$ is \textit{rigid} if $\Ascr_1(B) = \emptyset$ (or equivalently $\lnd(B) = \{0\}$),
and that $B$ is \textit{semi-rigid} if $ | \Ascr_1(B) | \le 1$.
\end{definition}  

\begin{remark}
Given a domain $B$ of caracteristic zero, $B$ is \textbf{not semi-rigid}
$\Leftrightarrow$  $ | \Ascr_1(B) | > 1$
$\Leftrightarrow$  $ | \Kscr_1(B) | > 1$
$\Leftrightarrow$  $ \trdeg(B : \ML(B)) > 1$
$\Leftrightarrow$  $ \trdeg(\Frac(B) : \FML(B)) > 1$.
\end{remark}

We now define a set $\Ascr_1^*(B)$ which is closely related to $\Ascr_1(B)$.
Note that $\Ascr_1^*(B)$ is defined for any integral domain $B$, of any characteristic.
The notation $B_A$ is defined at the end of the Introduction.

\begin{definition} \label {0n3it98bclsmbybp}
Given a domain $B$, we define $\Ascr_1^*(B)$ to be the set of subrings $A$ of $B$ such that $A$ is algebraically closed in $B$ and $B_A = (A_A)^{[1]}$.
\end{definition}

\begin{lemma} \label {0cvh2039edvjpqw0}
Let $B$ be a domain.
\begin{enumerata}

\item If $A_1,A_2 \in \Ascr_1^*(B)$ and $A_1 \subseteq A_2$ then $A_1 = A_2$.

\item If $A \in  \Ascr_1^*(B)$ then $B \cap \Frac(A) = A$ and $A$ is factorially closed in $B$.
Consequently, if $\bk$ is a field included in $B$ then $\bk \subseteq A$ for all $A \in  \Ascr_1^*(B)$.

\item If $\Char B = 0$ then $\Ascr_1(B) \subseteq \Ascr_1^*(B)$.
If moreover $B$ is not semi-rigid then $| \Ascr_1^*(B) | > 1$.

\item If $\Char B = 0$,  $A \in \Ascr_1^*(B)$ and $B$ is finitely generated as an $A$-algebra, then  $A \in \Ascr_1(B)$.

\end{enumerata}
\end{lemma}

\begin{proof}
Assertions (a) and (b) are clear and (c) follows from \ref{pc9293ed0wdjo03}(iv).
For (d), write $K = A_A$ and choose $t \in B$ such that $B_A = K[t] = K^{[1]}$.
Since $B$ is finitely generated as an $A$-algebra, there exists $r \in A \setminus \{0\}$ such that
the $K$-derivation $r \frac d{dt}: K[t] \to K[t]$ maps $B$ into itself.
Let $D : B \to B$ be the restriction of  $r \frac d{dt}$, then $D \in \lnd(B)$, $D \neq 0$ and $D(x)=0$ for all $x \in A$.
As $\trdeg_A(B)=1$ and $A$ is algebraically closed in $B$, we have $\ker(D)=A$, so (d) is proved.
\end{proof}

\begin{definition}
The \textit{height} of a poset $(X,\preceq)$ is the supremum of the set of $n \in \Nat$ for which there exists
a sequence $x_0 \prec x_1 \prec \cdots \prec x_n$  with $x_0, \dots, x_n \in X$.
We write $\haut(X)$ for the height of $(X,\preceq)$ and we regard $\haut(X)$ as an element of $\Nat \cup \{ \infty \}$.
\end{definition}

\begin{lemma} \label {8237ted7f983te}
Let $\tilde B$ be the normalization of a noetherian $\Rat$-domain $B$.
\begin{enumerata}

\item Each $D \in \lnd(B)$ has a unique extension to a locally nilpotent derivation $\tilde D : \tilde B \to \tilde B$.

\item Given a subset $\Delta$ of $\lnd(B)$, define $\tilde\Delta = \setspec{ \tilde D }{ D \in \Delta } \subseteq \lnd( \tilde B)$
and consider  $K_{ \tilde\Delta } \in \Kscr(\tilde B)$ and $A_{ \tilde\Delta } \in \Ascr(\tilde B)$.
Then $K_{ \tilde\Delta } = K_\Delta$ and $A_{ \tilde\Delta } \cap B = A_\Delta$ for all $\Delta \subseteq \lnd(B)$.

\item $\Kscr(B) \subseteq \Kscr(\tilde B)$

\item For each $A \in \Ascr(B)$, define $\Delta(A) = \setspec{ D \in \lnd(B) }{ A \subseteq \ker(D) }$.
Then $A \mapsto A_{ \widetilde{\Delta(A)} }$ is an injective order-preserving map $\Ascr(B) \to \Ascr(\tilde B)$.

\item $\haut\Ascr(B) \le \haut\Ascr(\tilde B)$ and $\haut\Kscr(B) \le \haut\Kscr(\tilde B)$.

\end{enumerata}
\end{lemma}

\begin{proof}
Part (a) is well known (Seidenberg's Theorem implies that each $D \in \lnd(B)$ extends (uniquely) to 
a derivation $\tilde D : \tilde B \to \tilde B$; then Vasconcelo's Theorem implies that $\tilde D$ is locally nilpotent).
To prove (b), consider $D \in \lnd(B) \setminus \{0\}$;
with $K = \Frac(\ker D)$ and $\tilde K = \Frac(\ker \tilde D)$, we have $K \subseteq \tilde K \subseteq \Frac B$
and (by \ref{pc9293ed0wdjo03}(iv)) $\Frac B = K^{(1)} = \tilde K^{(1)}$, so $K=\tilde K$.
This shows that for all $D \in \lnd(B)$ we have $\Frac( \ker \tilde D ) = \Frac( \ker D )$, and of course we also have $\ker(\tilde D) \cap B = \ker(D)$;
it follows that $K_{ \tilde\Delta } = K_\Delta$ and $A_{ \tilde\Delta } \cap B = A_\Delta$ for all $\Delta \subseteq \lnd(B)$, so (b) is true.
Assertion (c) follows from (b).
It is clear that $A \mapsto A_{ \widetilde{\Delta(A)} }$ is a well defined order-preserving map $\Ascr(B) \to \Ascr(\tilde B)$;
this map is injective because, by (b), we have $A_{ \widetilde{\Delta(A)} } \cap B = A_{ \Delta(A) } = A$ for all $A \in \Ascr(B)$; so (d) is true.
Since there exists an order-preserving injective map $\Ascr(B) \to \Ascr(\tilde B)$, we get $\haut\Ascr(B) \le \haut\Ascr(\tilde B)$;
$\haut\Kscr(B) \le \haut\Kscr(\tilde B)$ follows from (c), so we are done.
\end{proof}

\begin{lemma} \label {pc09fb23dcp90qj}
Let $K$ be a field and $v : K^* \to G$ a valuation of $K$, where $(G,+,\le)$ is a totally ordered abelian group.
Consider the field $K(x_1, \dots, x_n) = K^{(n)}$ and let $\Integ^n \times G$ be endowed with the lexicographic order.
Then there exists a valuation $\hat v : K(x_1, \dots, x_n)^* \to \Integ^n \times G$ such that
$\hat v(x_i) = ( \delta_{i,1}, \dots, \delta_{i,n}, 0)$ for $1 \le i \le n$ (where $\delta_{i,j}$ is the Kronecker delta)
and $\hat v( a ) = (0, \dots, 0, v(a) )$ for all $a \in K^*$.
\end{lemma}

\begin{proof}
First consider the case $n=1$. Note that $K(x) = K^{(1)}$ is a subfield of $K((x))$.
Given $f \in K(x)^*$, we may write $f = \sum_{i=m}^\infty a_i x^i$ with $a_i \in K$ for all $i$ and $a_m \neq 0$;
then we define $\hat v(f) = (m, v(a_m))$.  Details left to the reader.
For the general case, we note that $K(x_1, \dots, x_n) = K(x_2, \dots, x_n) (x_1)$ and argue by induction.
\end{proof}

\begin{definition}  \label {pcvie82p3sbfyaeop}
Let $\bk$ be a field, $\ck$ its algebraic closure and $R$ a $\bk$-domain.
\begin{enumerata}

\item We say that $R$ is \textit{rational over $\bk$} if $\Frac R$ is a purely transcendental extension of $\bk$.
We say that $R$ is \textit{unirational over $\bk$} if there exists a field $F$ such that  $\bk \subseteq \Frac R \subseteq F$ and
$F$ is a purely transcendental extension of $\bk$.

\item We say that $R$ is \textit{geometrically rational} (resp.\ \textit{geometrically unirational}) over $\bk$
if $\ck \otimes_\bk R$ is a domain and is rational (resp.\ unirational) over $\ck$.

\item We say that $R$ is \textit{absolutely factorial} if both $R$ and $\ck \otimes_\bk R$ are unique factorization domains.

\end{enumerata}
\end{definition}

\begin{remark}
The terms defined in parts (a) and (b) of Def.\ \ref{pcvie82p3sbfyaeop} may be used when $R$ is a field,
but we avoid using ``absolutely factorial'' for a field.
\end{remark}

\section{Locally nilpotent derivations and rationality}
\label {SectionSomeClarifications}

The literature devoted to locally nilpotent derivations and rationality contains a few claims that have been published with invalid proofs.
Because those claims are directly at the center of the subject of the present article,
we feel that it is necessary to provide proofs for them before we can go forward with this investigation.
It is the aim of this section to provide such proofs.
In some cases we strengthen the claim being considered, and in one case
(Prop.\ \ref{cxvoq9wedowjskdpx}\eqref{cxvoq9wedo-a}) we give a simpler proof for a result whose published proof seems to be correct.
We organize this discussion more or less in the form of a historical account, but our goal is not to be exhaustive from the historical point of view;
we simply want to cover the topics that require clarification.

%%%%%%%%%%%%%%%%%%%%%%%%%%%%%%%%%%%%%%%%%%%%%%%%%%%%%%%%%%%%%%%%%%%%%%%%%%%%%%%%%%%%%%%%%%%%%%%%%%%%%%%%%%%%%%%%%%%%%%%%%%%%%%%%%%%

Early work in this area was influenced by the question whether the implication
\begin{equation} \label {pcfh2939edxpX0djf982}
\ML(B) = \bk \implies  \textit{$B$ is rational over $\bk$}
\end{equation}
is true or false, where it is assumed that $B$ is an affine $\bk$-domain and that $\bk$ is an algebraically closed field of characteristic zero.
The implication was known to be true when $\dim B \le 2$ but the general case remained open for several years.
Then Liendo showed that, for every integer $d\ge3$, there exists a counterexample $B$ to implication \eqref{pcfh2939edxpX0djf982} with $\dim B = d$
(cf.\  \cite[Lemma 4.4]{Liendo_Rationality2010}).
In \cite[Ex.\ 1.22]{Popov_Russellfest}, Popov showed that for every integer $d\ge3$ there exist counterexamples with $\dim B \ge d$
and such that the variety $\Spec B$ is smooth.
In both cases (Liendo and Popov), understanding the counterexamples requires familiarity with some sophisticated geometry.
As far as we know, no simple proof has been circulated for such examples.
Our first result gives counterexamples to  \eqref{pcfh2939edxpX0djf982} and the proof of 
part \eqref{cxvoq9wedo-a}  is particularly simple.

\begin{proposition} \label {cxvoq9wedowjskdpx}
Let $K/\bk$ be a finitely generated extension of fields of characteristic zero.
\begin{enumerata}

\item \label {cxvoq9wedo-a}  There exists an affine $\bk$-domain $B$ satisfying
\begin{equation*} 
\text{$\Frac B = K^{(2)}$, $\ML(B) = \bk$ and $K \in \Kscr(B)$.}
\end{equation*}

\item \label {cxvoq9wedo-b}  There exists a \textbf{normal} affine $\bk$-domain $B$ satisfying
\begin{equation*} 
\text{$\Frac B = K^{(2)}$, $\ML(B) = \bk'$ and  $K \in \Kscr(B)$}
\end{equation*}
where $\bk'$ is the algebraic closure of $\bk$ in $K$.

\end{enumerata}
\end{proposition}

\begin{proof}
\eqref{cxvoq9wedo-a}
Choose $r_1, \dots, r_m$ such that $K = \bk( r_1, \dots, r_m )$ and choose $x,y$ such that $K(x,y) = K^{(2)}$.
Let $B = \bk[x,y, \, r_1x, r_1y, \, r_2x, r_2y, \, \dots, \,  r_m x, r_m y ] \subseteq K[x,y]$.
Then $B$ is an affine $\bk$-domain such that $\Frac B = K(x,y)=K^{(2)}$ and $B \cap K = \bk$.
Consider the $K$-derivations $\delta_1 = y \frac{\partial}{\partial x}$ and $\delta_2 = x \frac{\partial}{\partial y}$ of $K[x,y]$.
Note that $\delta_1$ and $\delta_2$ are locally nilpotent and map $B$ into itself.
Let $D_1, D_2 : B \to B$ be the restrictions of $\delta_1,\delta_2$ respectively, then $D_1,D_2 \in \lnd(B)$ and
$\ker(D_1) \cap \ker(D_2) = B \cap K =\bk$, so $ \ML(B)=\bk $.
It is easy to see that $K = \Frac(\ker D_1) \cap \Frac(\ker D_2)$, so $K \in \Kscr(B)$.

\eqref{cxvoq9wedo-b}
Let $B = \bk[x,y, \, r_1x, r_1y, \, r_2x, r_2y, \, \dots, \,  r_m x, r_m y ]$ be the ring defined in the above paragraph and
let $\tilde B$ be the normalization of $B$.
Then $\tilde B$ is a normal affine $\bk$-domain and $\Frac( \tilde B ) = \Frac(B) = K^{(2)}$.
It is also clear that $K \in \Kscr( \tilde B )$, because $K \in \Kscr(B) \subseteq \Kscr( \tilde B )$ (see Lemma \ref{8237ted7f983te}).
So, to prove the claim, it suffices to show that
\begin{equation} \label {c0vh230efXdfj}
\ML( \tilde B ) = \bk' .
\end{equation}
Since $\tilde B$ is normal, we have $\bk' \subseteq \tilde B$ and consequently $\bk' \subseteq \ML( \tilde B )$ by \ref{pc9293ed0wdjo03}(i).
To prove that $\ML( \tilde B ) \subseteq \bk'$, consider $\alpha \in \ML( \tilde B ) \setminus \{0\}$.
Then $\alpha \in K$ is clear (because $K \in \Kscr( \tilde B )$) and it suffices to show that $\alpha$ is algebraic over $\bk$.
So it's enough to check that $v(\alpha)\ge0$ for every valuation $v$ of $K/\bk$.

Consider an arbitrary valuation $v : K^* \to G$ of $K$ over $\bk$.
Recall that $\Frac \tilde B = K(x,y) = K^{(2)}$.
By Lemma \ref{pc09fb23dcp90qj}, there exists a valuation $\hat v : K(x,y)^* \to \Integ^2 \times G$ satisfying
$\hat v(x) = (1,0,0)$, $\hat v(y) = (0,1,0)$ and $\hat v(a) = (0,0,v(a))$ for all $a \in K^*$.
As $\alpha \in K^*$, we have $\hat v(\alpha) = (0,0,v(\alpha))$.
We have $\hat v(\xi) \ge (0,0,0)$ for each $\xi \in \bk \cup \{x,y, \, r_1x, r_1y, \, r_2x, r_2y, \, \dots, \,  r_m x, r_m y \}$,
so   $\hat v(\xi) \ge (0,0,0)$ for all $\xi \in B$ and hence for all $\xi \in \tilde B$.
As $\alpha \in \tilde B$, $(0,0,0) \le \hat v(\alpha) = (0,0,v(\alpha))$, so $v(\alpha) \ge 0$.
Since this is true for every valuation $v$ of $K/\bk$, $\alpha$ is algebraic over $\bk$.
This proves \eqref{c0vh230efXdfj}, and completes the proof of the Proposition.
\end{proof}

Prop.\ \ref{cxvoq9wedowjskdpx}\eqref{cxvoq9wedo-a}  implies the last assertion of \cite[Thm 4.2]{Liendo_Rationality2010};
however the proof given in \cite{Liendo_Rationality2010} is much more complicated than this one. 
Also note that  Prop.\ \ref{cxvoq9wedowjskdpx}\eqref{cxvoq9wedo-b} gives us normal counterexamples to \eqref{pcfh2939edxpX0djf982} in every dimension $\ge3$.
More precisely:
\begin{enumerate}
\item[$(*)$] \begin{minipage}[t]{.9\linewidth} \it
If $\bk$ is a field of characteristic zero and $d\ge3$ is an integer, then there exists a normal affine $\bk$-domain $B$
such that $\dim B = d$, $\ML(B)=\bk$ and $B$ is not unirational (hence not rational) over $\bk$.
\end{minipage}
\end{enumerate}
Indeed, let $K$ be a finitely generated extension of $\bk$ such that $\trdeg_\bk(K) = d-2$,
$K$ is not unirational over $\bk$ and $\bk$ is algebraically closed in $K$.\footnote{For instance,
let $F$ be the field of fractions of $\bk[u,v] / (v^2 - u(u^2-1))$ (where $\bk[u,v] = \kk2$) and set $K = F^{(d-3)}$.}
Applying Prop.\ \ref{cxvoq9wedowjskdpx}\eqref{cxvoq9wedo-b} to $K/\bk$ gives a domain $B$ satisfying the requirements of $(*)$.

\bigskip

The following implication is also considered in 
\cite{Liendo_Rationality2010}, for a field $\bk$ of characteristic zero and an affine $\bk$-domain $B$:
\begin{equation} \label {p9cv9p23oe0fpdof}
\textit{$\dim B \ge 2$ and $\ML(B) = \bk$} \ \  \implies \ \  \textit{$\Frac B = K^{(2)}$ for some extension field $K$ of $\bk$}.
\end{equation}
In fact the first part of \cite[Thm 4.2]{Liendo_Rationality2010}
asserts that \eqref{p9cv9p23oe0fpdof} is true.
However, the proof given in  \cite{Liendo_Rationality2010} is based on the following false statement:
{\it if $D_1,D_2 \in \lnd(B)\setminus\{0\}$ satisfy $\ker(D_1) \neq \ker(D_2)$, then $\Frac B$ has transcendence degree $2$
over $\Frac(\ker D_1) \cap \Frac(\ker D_2)$} (examples show that this transcendence degree can be larger than $2$).
A different proof of implication \eqref{p9cv9p23oe0fpdof} is given in \cite{Dub_Liendo_2016}, but it is also invalid.
More precisely, implication \eqref{p9cv9p23oe0fpdof} would follow from \cite[Cor.\ 3.2]{Dub_Liendo_2016}, which is itself
a consequence of \cite[Prop.\ 3.1]{Dub_Liendo_2016}.
However, the proof of \cite[Prop.\ 3.1]{Dub_Liendo_2016} is faulty:
one has two subfields  $K_{X \times \proj^1}^{G_a}$ and $K_X$ of a field $K_{X \times \proj^1}$ and 
one has to prove the inclusion $K_{X \times \proj^1}^{G_a} \subseteq K_X$;
it is claimed that the inclusion is proved, but in fact the argument only
proves that $K_{X \times \proj^1}^{G_a}$ is isomorphic to a subfield of $K_X$,
which is not sufficient for the proof to be valid.\footnote{I discussed
these issues with the authors of \cite{Liendo_Rationality2010} and \cite{Dub_Liendo_2016} and
they agree that the problems that I am pointing out make their proofs invalid.
At the time of those discussions, I didn't know how to fix the problem; Cor.\ \ref{0xc9vj2o3w9edno9} occurred to me a year later.
Cor.\ \ref{0xc9vj2o3w9edno9} implies that the statement of \cite[Thm 4.2]{Liendo_Rationality2010} is true,
but I don't know if that of \cite[Prop.\ 3.1]{Dub_Liendo_2016} is true or false.}

Nevertheless, the following result ascertains that implication \eqref{p9cv9p23oe0fpdof} is true
(the implication follows from Cor.\ \ref{0xc9vj2o3w9edno9}).
Recall that a field extension $L/E$ is said to be ruled (one also says that $L$ is ruled over $E$)
if there exists a field $F$ such that $E \subseteq F \subseteq L$ and $L = F^{(1)}$.

\begin{proposition} \label {ckjv9hrjdulxjkdxlkbrvn}
Let $\bk$ be a field and $B$ a $\bk$-domain such that $\trdeg_\bk(B) < \infty$ and $| \Ascr_1^*(B) | > 1$.
Then the following hold.
\begin{enumerata}

\item For each $A \in \Ascr_1^*(B)$, we have $\bk \subseteq A$ and $\Frac(A)$ is ruled over $\bk$.

\item There exists a field $K$ satisfying $\bk \subseteq K \subseteq \Frac B$ and $\Frac B = K^{(2)}$.

\end{enumerata}
\end{proposition}

\begin{proof}
In (a), the fact that $\bk \subseteq A$ follows from Lemma \ref{0cvh2039edvjpqw0}.
To prove the other part of (a), consider distinct elements $A_1, A_2$ of $\Ascr_1^*(B)$, let $F_i = \Frac A_i$ ($i=1,2$) and let us prove 
that $F_1$ is ruled over $\bk$.
Choose $s_1,s_2 \in B$ such that (for each $i$) $B_{A_i} = F_i[ s_i ] = F_i^{[1]}$.
Let $v$ be the valuation of $\Frac B = F_1(s_1)$ over $F_1$ that satisfies $v(s_1) = -1$,
and let $v_0$ be the valuation of $F_2$ obtained by restricting $v$.
If $\kappa$ (resp.\ $\kappa_0$) denotes the residue field of $v$ (resp.\ of $v_0$) then we have the field extensions $\kappa / \kappa_0 / \bk$.
Since $\Frac B = F_2^{(1)}$, the Ruled Residue Theorem \cite[Lemma 2.3]{OhmSurvey} implies that the extension $\kappa/\kappa_0$ is either ruled or algebraic.
Note that $A_2 \nsubseteq A_1$ (see Lemma \ref{0cvh2039edvjpqw0}) and choose $b \in A_2 \setminus A_1$.
Let $d$ be the degree of $b \in B \subseteq F_1[s_1]$ as a polynomial in $s_1$;
then $d>0$, otherwise $b \in F_1 \cap B = A_1$ (by Lemma \ref{0cvh2039edvjpqw0}), which is not the case.
So $v(b) = -d <0$; as $b \in F_2^*$, $v_0(b)$ is defined and $v_0(b) = v(b) <0$.
So $v_0$ is not the trivial valuation on $F_2$; since $\trdeg_\bk(B)<\infty$, it follows that $\trdeg_\bk(\kappa_0) < \trdeg_\bk(F_2)$.
Now $\trdeg_\bk(F_2) = \trdeg_\bk(B)-1 = \trdeg_\bk(F_1) = \trdeg_\bk(\kappa)$, the last equality because $\kappa$ is $\bk$-isomorphic to $F_1$,
so we obtain $\trdeg_\bk(\kappa_0) < \trdeg_\bk(\kappa)$, showing that $\kappa / \kappa_0$ is not algebraic. So $\kappa / \kappa_0$ is ruled,
and consequently $\kappa/\bk$ is ruled.
Since $F_1$ is $\bk$-isomorphic to $\kappa$ it follows that $F_1$ is ruled over $\bk$.
This proves (a).

For (b), choose $A \in \Ascr_1^*(B)$; then (a) implies that there exists $K$ such that $\bk \subseteq K \subset \Frac(A)$ and $\Frac(A) = K^{(1)}$.
Since $\Frac(B) = (\Frac A)^{(1)}$, (b) follows.
\end{proof}

\begin{corollary} \label {0xc9vj2o3w9edno9}
Let $\bk$ be a field of characteristic zero and $B$ a $\bk$-domain.
If  $\trdeg_\bk(B) < \infty$ and $B$ is not semi-rigid then the following hold.
\begin{enumerata}

\item \label {dof239f2p0bfe7} For each $D \in \lnd(B)$, $\Frac(\ker D)$ is ruled over $\bk$.
\item \label {c0kwj4bt9fp23k}   There exists a field $K$ satisfying $\bk \subseteq K \subseteq \Frac B$ and $\Frac B = K^{(2)}$.

\end{enumerata}
\end{corollary}

\begin{proof}
We have $\Ascr_1(B) \subseteq \Ascr_1^*(B)$ by	Lemma \ref{0cvh2039edvjpqw0},
so  $| \Ascr_1^*(B) | > 1$, 
so the claim follows from Prop.\ \ref{ckjv9hrjdulxjkdxlkbrvn}.
\end{proof}

It follows from Cor.\ \ref{0xc9vj2o3w9edno9} that the statement of \cite[Thm 4.2]{Liendo_Rationality2010} is true,
despite the fact that the proofs of it given in \cite{Liendo_Rationality2010} and \cite{Dub_Liendo_2016} are flawed.
Note that Liendo presented his result in the form of a birational characterization of affine $\bk$-varieties $X$ satisfying $\ML(X)=\bk$.
That is an interesting viewpoint, so let us reformulate our results as follows:

\begin{corollary} \label {co9fp2o93ibrfp0er9yflaFerif}
Let $L/\bk$ be a finitely generated extension of fields of characteristic zero such that $\trdeg_\bk L \ge2$.
Then the following are equivalent:
\begin{enumerata}

\item There exists an affine $\bk$-domain $B$ satisfying  $\ML(B)=\bk$ and $\Frac B = L$;

\item there exists an affine $\bk$-domain $B$ that is not semi-rigid and satisfies $\Frac B = L$;

\item there exists a field $K$ satisfying $\bk \subseteq K \subseteq L$ and $L = K^{(2)}$.

\end{enumerata}
\end{corollary}

\begin{proof}
Implication (a)$\Rightarrow$(b) is trivial.  Implications (b)$\Rightarrow$(c) and (c)$\Rightarrow$(a)
follow, respectively, from Cor.\ \ref{0xc9vj2o3w9edno9}\eqref{c0kwj4bt9fp23k}   and Prop.\ \ref{cxvoq9wedowjskdpx}\eqref{cxvoq9wedo-a}.
\end{proof}

\medskip

It is in Liendo's thesis that the definition of $\FML(B)$ was first proposed, together with the conjecture
that $\FML(B) = \bk$ implies rationality of $B$ over $\bk$.
Then Arzhantsev, Flenner, Kaliman, Kutzschebauch, Zaidenberg \cite[Prop.\ 5.1]{Arz_Flen_Kaliman_Kutz_Zaid:FlexAut} 
and Popov \cite[Thm 4]{Popov_InfiniteDimAlgGps2014} proved:

\begin{theorem} \label {kfkjep930120wdkoi9}
If $\bk$ is an algebraically closed field of characteristic zero and $B$ is an affine $\bk$-domain
satisfying $\FML(B) = \bk$, then $B$ is unirational over $\bk$.
\end{theorem}

\begin{remark}
Suppose that $\bk$ and $B$ satisfy all assumptions of Thm \ref{kfkjep930120wdkoi9} and let $n = \dim B$.
Then one can show that $B \subseteq \bk^{[n]}$, which is stronger than $B$ being unirational over $\bk$.
See Thm \ref{9i3oerXfvdf93p04efeJ}.
\end{remark}

Observe that Thm \ref{kfkjep930120wdkoi9} does not answer the question whether the implication
\begin{equation} \label{8263fnehqwy1jk} 
\textit{$\FML(B) = \bk$ $\implies$ $B$ is rational over $\bk$}
\end{equation}
is true or false, where $\bk$ is an algebraically closed field of characteristic zero and $B$ is an affine $\bk$-domain.
It is well known that \eqref{8263fnehqwy1jk} is true when $\dim B \le 2$, even without assuming that $\bk$ is algebraically closed
(see Cor.\ \ref{cknvlwr90fweodpx}(a)).
On the other hand, Popov \cite[Thm 2]{Popov_RatFML2013} showed that there exist affine $\bk$-domains $B$ satisfying $\FML(B)=\bk$,
$B$ is not stably rational over $\bk$ and $\Spec B$ is a smooth variety, and any such $B$ is a counterexample to \eqref{8263fnehqwy1jk}.
As these examples of Popov all satisfy $\dim B \ge 263168$,\footnote{In the smallest example of Popov,
$\Spec B$ is a quotient of $\mathbf{SL}_{513}$ by a finite subgroup.}
it is interesting to discuss what happens for small values of $\dim B \ge 3$.

Result \cite[Thm 5.6]{Liendo_Rationality2010} asserts
that implication \eqref{8263fnehqwy1jk} is true when $\dim B \le 3$, but the proof given there is invalid for several reasons.
(One of the reasons is that the proof of \cite[Thm 5.6]{Liendo_Rationality2010} uses  \cite[Lemma 5.4]{Liendo_Rationality2010},
whose proof is faulty: on the first line of page 3665, it is not true that $L \cap \Frac(\ker\partial)=\bk$
together with Lemma 1.3(i) imply that $\Frac(\ker\partial)=\bk(x_1,\dots,x_n)$.)

We now prove that implication \eqref{8263fnehqwy1jk} is true at least up to dimension 4:

\begin{corollary} \label {cknvlwr90fweodpx}
Let $\bk$ be a field of characteristic zero and $B$ an affine $\bk$-domain satisfying \mbox{$\FML(B) = \bk$.}
\begin{enumerata}

\item If $\dim B \le 2$ then  $B$ is rational over $\bk$.

\item If $\dim B \le 4$ and $\bk$ is algebraically closed then $B$ is rational over $\bk$.

\end{enumerata}
\end{corollary}

\begin{proof}
Let $n = \dim B$.  We may assume that $n \in \{2, 3, 4\}$, otherwise the result is trivial
(in fact part (a) is well known, but we include a proof of the case $n=2$ because it is short).
By Cor.\ \ref{0xc9vj2o3w9edno9}, there exists a field $K$ such that $\bk \subseteq K \subseteq \Frac B$ and $\Frac B = K^{(2)}$.
If $n=2$ then we must have $K=\bk$ (because $\FML(B) = \bk$ implies that $\bk$ is algebraically closed in $\Frac B$),
so $\Frac B = \bk^{(2)}$.
Assume that $n \in \{3, 4\}$ and that $\bk$ is algebraically closed.
By Thm \ref{kfkjep930120wdkoi9}, $\Frac(B)$ is unirational over $\bk$, so $K/\bk$ is unirational.
If $n=3$ then $K=\bk^{(1)}$ by the Generalized L\"{u}roth Theorem;
if $n=4$ then $K=\bk^{(2)}$ by Castelnuovo's Theorem (for instance Remark 6.2.1, p.\ 422 of \cite{Hartshorne}).
So in all cases we have $K = \bk^{(n-2)}$, so  $\Frac B = K^{(2)} = \bk^{(n)}$.
\end{proof}

\section{Properties of $\Ascr(B)$ and $\Kscr(B)$}
\label {Somepreliminaries}

Most of this section is devoted to establishing the basic properties of the posets $\Ascr(B)$ and $\Kscr(B)$ introduced in
Def.\ \ref{0cd20FXGICHGgt86cjF5i5rgo909}. The results obtained here are used in the subsequent sections.
At the end of the present section we introduce another invariant of rings, $\lndrank(B)$, also defined in terms of locally nilpotent derivations.

\begin{lemma} \label {0ckj238hqn239hnfpawrnHzb}
If $B$ is a domain of characteristic zero then
each element of $\Ascr(B)$ is factorially closed in $B$ and 
each element of $\Kscr(B)$ is algebraically closed in $\Frac B$.
In particular, $\FML(B)$ is algebraically closed in $\Frac B$.
\end{lemma}

\begin{proof}
By paragraph \ref{pc9293ed0wdjo03}, if $D \in \lnd(B)$ then $\ker(D)$ is factorially closed in $B$
and $\Frac( \ker D )$ is algebraically closed in $\Frac B$. The result follows.
\end{proof}

\begin{remark}
Throughout the article we make tacit use of the following fact, which is a consequence of Lemma \ref{0ckj238hqn239hnfpawrnHzb}:
if $R, A \in \Ascr(B)$ satisfy $R \subset A$  then $\trdeg_R(A)>0$,
and similarly, if $K, L \in \Kscr(B)$ satisfy $K \subset L$  then $\trdeg_K(L)>0$
(recall that ``$\subset$'' means strict inclusion).
\end{remark}

\begin{lemma} \label {f2i3p0f234hef}
Let $B$ be a domain of characteristic zero.
\begin{enumerata}

\item Let $D \in \lnd(B)$ and let $D' \in \Der( \Frac B )$ be the unique extension of $D$ to a derivation of $\Frac B$.
Then $\Frac(\ker D) = \ker D'$.

\item We have $B \cap K_\Delta = A_\Delta$ for all subsets $\Delta$ of $\lnd(B)$.
In particular, $B \cap \FML(B) = \ML(B)$.

\item Given $A \in \Ascr(B)$, define $\Delta(A) = \setspec{ D \in \lnd(B) }{ A \subseteq \ker D }$.
Then $K_{\Delta(A)} \in \Kscr(B)$ and $B \cap K_{\Delta(A)} = A$.

\item The maps
$f : \Kscr(B) \to \Ascr(B)$, $K \mapsto B \cap K$
and
$g : \Ascr(B) \to \Kscr(B)$, $A \mapsto K_{\Delta(A)}$
are well defined and $f \circ g$ is the identity map of $\Ascr(B)$.
In particular, $f$ is surjective and $g$ is injective.

\end{enumerata}
\end{lemma}

\begin{proof}
(a) We may assume that $D \neq 0$, otherwise the claim is trivial.
Let $L=\Frac B$ and $K = \Frac(\ker D)$; then $L = K^{(1)}$ by \ref{pc9293ed0wdjo03}.
If $a,a' \in \ker D$ ($a' \neq 0$) then $D'(a/a') = (D(a) a' - a D(a'))/(a')^2 = 0$, showing that $K \subseteq \ker D'$.
If the inclusion $K \subseteq \ker D'$ is strict then $L$ is algebraic over $\ker D'$ (because $L = K^{(1)}$), so $D'=0$, contradicting $D \neq 0$.
So $K=\ker D'$.

(b) By assertion (a),
we have $B \cap \Frac( \ker D ) = \ker D$ for each $D \in \lnd(B)$. So
$$
\textstyle
B \cap K_\Delta
= B \cap \bigcap_{D \in \Delta} \Frac( \ker D )
= \bigcap_{D \in \Delta} (B \cap \Frac( \ker D ))
= \bigcap_{D \in \Delta} \ker(D) = A_\Delta.
$$

(c) If $A \in \Ascr(B)$ then $B \cap K_{\Delta(A)} = A_{\Delta(A)}$ by (b) and  $A_{\Delta(A)} = A$  is clear.

Assertion (d) follows.
\end{proof}

\begin{remark}
$\xymatrix@1{(\Kscr(B), \subseteq) \ar @<.4ex>[r]^{f} &   (\Ascr(B), \subseteq) \ar @<.4ex>[l]^{g}  }$ are order-preserving maps between posets.
\end{remark}

\begin{lemma} \label {0c9vni3w4e8vno}
For any domain $B$ of characteristic zero we have
$\haut \big( \Ascr(B) \big) \le \haut \big( \Kscr(B) \big)$.
\end{lemma}

\begin{proof}
By Lemma \ref{f2i3p0f234hef}, there exists an injective order-preserving map 
$(\Ascr(B), \subseteq) \to (\Kscr(B), \subseteq)$.
So $\haut \big( \Ascr(B) \big) \le \haut \big( \Kscr(B) \big)$.
\end{proof}

The following shows that the maps $f,g$ of Lemma \ref{f2i3p0f234hef} are not necessarily bijective, even when $B$ is a normal domain.

\begin{example}
Let $\bk$ be a field of characteristic zero and consider the subalgebra $B = \bk[x,y,tx,ty]$ of  $\bk[x,y,t] = \kk3$.
Then $B$ is normal (define a $\Integ$-grading $\bk[x,y,t] = \bigoplus_{n \in \Integ} R_n$ by declaring that $\bk \subseteq R_0$,
$x,y \in R_1$ and $t \in R_{-1}$; then $B = \bigoplus_{n \ge 0} R_n$, so $B$ is integrally closed in $\bk[x,y,t]$, so $B$ is normal).
Define $D_1,D_2,D_3 \in \lnd(B)$ by $D_1 =y\frac{\partial}{\partial x}$,  $D_2 = x\frac{\partial}{\partial y}$ and $D_3 = \frac{\partial}{\partial t}$.
We have $\bk(y,t), \bk(x,t), \bk(x,y) \in \Kscr(B)$ (these are $\Frac(\ker D_i)$ for $i=1,2,3$),
so $\bk(x), \bk(y), \bk(t) \in \Kscr(B)$ and $\FML(B) = \bk$.
Note that $\bk(t)$ and $\bk$ are distinct elements of $\Kscr(B)$ that have the same image under the map $f$ of Lemma \ref{f2i3p0f234hef},
so $f,g$ are not bijective.
\end{example}

Also note that $f,g$ do not necessarily preserve transcendence degree, even when $B$ is normal and $f,g$ are bijective
(see Ex.\ \ref{pcivn2309efdfXe}).
However, the next result states that if $B$ is a UFD then $f,g$ are isomorphisms of posets
$\xymatrix@1{(\Kscr(B), \subseteq) \ar @<.4ex>[r] &   (\Ascr(B), \subseteq) \ar @<.4ex>[l]  }$
\textit{and} preserve transcendence degree.

\begin{lemma} \label {9vdyf25fjmg752rd}
Let $B$ be a UFD of characteristic zero.  Consider the maps $f$ and $g$ of Lemma \ref{f2i3p0f234hef}.
\begin{enumerata}

\item For any subset $\Delta$ of $\lnd(B)$, we have $\Frac( A_\Delta ) = K_\Delta$.
In particular, $\Frac\big( \ML(B) \big) = \FML(B)$.

\item The map $g : \Ascr(B) \to \Kscr(B)$ satisfies $g(A) = \Frac(A)$ for all $A \in \Ascr(B)$.

\item The maps $f$ and $g$ are bijective and inverse of each other.
Moreover, if $K \in \Kscr(B)$ and $A \in \Ascr(B)$ satisfy $f(K)=A$ and $g(A)=K$ then $\trdeg_K(\Frac B) = \trdeg_A(B)$.

\end{enumerata}
\end{lemma}

\begin{proof}
(a) Consider a subset $\Delta$ of $\lnd(B)$. It is clear that $\Frac( A_\Delta ) \subseteq K_\Delta$, so let us prove the reverse inclusion.
Let $\xi \in K_\Delta$. 
Write $\xi = u/v$ with $u,v \in B$, $v \neq0$, and $\gcd(u,v)=1$.

Let $D \in \Delta$.  The extension $D' \in \Der( \Frac B )$ of $D$ satisfies $\ker(D') = \Frac( \ker D ) \supseteq K_\Delta \ni \xi$, so
$0 = D'(u/v) = \frac{ v D(u) - u D(v) }{ v^2 }$, so $v D(u) = u D(v)$. Since $u \mid v D(u)$ and $\gcd(u,v)=1$ we have $u \mid D(u)$,
so $D(u)=0$ by \ref{pc9293ed0wdjo03}(v); similarly, $v \mid D(v)$, so $D(v)=0$.
This shows that $u,v \in \ker(D)$, and this holds for an arbitrary $D \in \Delta$.
So $u,v \in A_\Delta$ and hence $\xi \in \Frac( A_\Delta )$. This proves that $\Frac( A_\Delta ) = K_\Delta$.
It follows that  $\Frac\big( \ML(B) \big) = \FML(B)$.

(b) If $A \in \Ascr(B)$ then $g(A) = K_{ \Delta(A) } = \Frac( A_{ \Delta(A) } ) = \Frac(A)$, where the middle equality follows from (a).

(c) Let $K \in \Kscr(B)$; then $K = K_{\Delta}$ for some $\Delta$; then $K = K_{\Delta} = \Frac( A_\Delta ) = g(A_\Delta)$ shows
that $g$ is surjective. As $f \circ g = \id$, $f$ and $g$ are bijective and inverse of each other.
The equality $\trdeg_K(\Frac B) = \trdeg_A(B)$ follows from $\Frac(A) = g(A) = K$.
\end{proof}

\begin{notation} \label {cov90vb6h4s}
For a domain $B$ of characteristic zero, 
we define
$$
T_\Ascr(B) = \setspec{ n \in \Nat }{ \Ascr_n(B) \neq \emptyset } \text{\ \ and\ \ } T_\Kscr(B) = \setspec{ n \in \Nat }{ \Kscr_n(B) \neq \emptyset }.
$$
\end{notation}

\begin{bigremark} \label {09bn349efvkmtZr8}
Let $B$ be a domain of characteristic zero. 
Clearly, the order-preserving maps \newline
$\xymatrix@1@C=13pt{(\Kscr(B), \subseteq) \ar @<.4ex>[r] &   (\Ascr(B), \subseteq) \ar @<.4ex>[l]  }$
of Lemma \ref{f2i3p0f234hef} restrict to bijections 
$\xymatrix@1@C=13pt{ \Kscr_n(B) \ar @<.4ex>[r] &   \Ascr_n(B) \ar @<.4ex>[l]  }$ for each $n \in \{0,1\}$
(see the last part of Def.\ \ref{0cd20FXGICHGgt86cjF5i5rgo909}).
By Lemma \ref{9vdyf25fjmg752rd}, if $B$ is a UFD then these maps restrict to bijections
$\xymatrix@1@C=13pt{ \Kscr_n(B) \ar @<.4ex>[r] &   \Ascr_n(B) \ar @<.4ex>[l]  }$ for all $n \in \Nat$, and 
consequently $T_\Ascr(B) = T_\Kscr(B)$.
\end{bigremark}

\begin{lemma} \label {pc0WVEDFWc9vn20wdfc8wgd9s8}
Let $B$ be a domain of characteristic zero.
If $| T_\Kscr(B) | \le 3$ then the following hold.
\begin{enumerata}

\item The order-preserving maps $\xymatrix@1@C=13pt{(\Kscr(B), \subseteq) \ar @<.4ex>[r] &   (\Ascr(B), \subseteq) \ar @<.4ex>[l]  }$
of Lemma \ref{f2i3p0f234hef} are isomorphisms of posets.\footnote{We are not claiming that they preserve transcendence degree.}

\item If $| T_\Kscr(B) | =1$ then $T_\Ascr(B) = \{0\} = T_\Kscr(B)$.

\item If $| T_\Kscr(B) | =2$ then $T_\Ascr(B) = \{0,1\} = T_\Kscr(B)$.

\item If $| T_\Kscr(B) | =3$ then $T_\Ascr(B) = \{0,1,n\}$ and  $T_\Kscr(B) = \{0,1,m\}$ where $1<m\le n$.

\end{enumerata}
\end{lemma}

\begin{proof}
We may assume that $| T_\Kscr(B) | =3$, otherwise all claims are trivial.
Then $T_\Kscr(B) = \{0,1,m\}$ where $m = \trdeg( \Frac(B) : \FML(B) ) > 1$.
It follows that $\haut( \Kscr(B) ) =2$, so Lemma \ref{0c9vni3w4e8vno} gives  $\haut( \Ascr(B) ) \le 2$.
Let $n = \trdeg( B : \ML(B) ) \ge m$, then $T_\Ascr(B) \supseteq \{0,1,n\}$.
If $T_\Ascr(B) \neq \{0,1,n\}$ then choose $i \in  T_\Ascr(B) \setminus \{0,1,n\}$
and $A_i \in \Ascr_i(B)$.  Note that $1<i<n$ and that there exists $A_1 \in \Ascr_1(B)$ such that $A_i \subset A_1$.
Then $\ML(B) \subset A_i \subset A_1 \subset B$ is a chain in $\Ascr(B)$, which contradicts $\haut( \Ascr(B) ) \le 2$.
Thus $T_\Ascr(B) = \{0,1,n\}$.
It follows that  we have the disjoint unions
$$
\Kscr(B) = \Kscr_0(B) \cup   \Kscr_1(B) \cup  \Kscr_m(B) \quad \text{and} \quad \Ascr(B) = \Ascr_0(B) \cup   \Ascr_1(B) \cup \Ascr_n(B)  
$$
where  $\Kscr_m(B) = \{ \FML(B) \}$ and $\Ascr_n(B) = \{ \ML(B) \}$ are singletons.
We noted in Rem.\ \ref{09bn349efvkmtZr8} that the maps $\xymatrix@1@C=13pt{\Kscr(B) \ar @<.4ex>[r] & \Ascr(B) \ar @<.4ex>[l]  }$ restrict to bijections
$\xymatrix@1@C=13pt{\Kscr_0(B) \ar @<.4ex>[r] & \Ascr_0(B) \ar @<.4ex>[l]  }$ and
$\xymatrix@1@C=13pt{\Kscr_1(B) \ar @<.4ex>[r] & \Ascr_1(B) \ar @<.4ex>[l]  }$,
and clearly they also restrict to bijections
$\Kscr_m(B) = \xymatrix@1@C=13pt{ \{ \FML(B) \} \ar @<.4ex>[r] & \{ \ML(B) \} = \Ascr_n(B) \ar @<.4ex>[l]  }$,
so  the maps $\xymatrix@1@C=13pt{\Kscr(B) \ar @<.4ex>[r] & \Ascr(B) \ar @<.4ex>[l]  }$ are bijective and we are done.
\end{proof}

The following shows that the maps $f,g$ of Lemma \ref{f2i3p0f234hef} 
do not necessarily preserve transcendence degree, even when $B$ is normal and $f,g$ are bijective.

\begin{example} \label {pcivn2309efdfXe}
Let $K / \Comp$ be the function field of a non-rational complex algebraic curve.
By  Prop.\ \ref{cxvoq9wedowjskdpx}\eqref{cxvoq9wedo-b}, there exists a normal affine $\Comp$-domain $B$ satisfying
$\Frac(B) = K^{(2)}$ (so $\dim B = 3$), $\ML(B) = \Comp$ and $K \in \Kscr(B)$.
Since $K/\Comp$ is not unirational, it follows that $B$ is not unirational over $\Comp$, so $\FML(B) \neq \Comp$ by Thm \ref{kfkjep930120wdkoi9}.
Consequently, $\FML(B) = K$ and hence $T_\Kscr(B) = \{0,1,2\}$.
By Lemma \ref{pc0WVEDFWc9vn20wdfc8wgd9s8}, $f$ and $g$ are bijective.
Since $\FML(B) \cap B = \ML(B)$, we get $f(K) = K \cap B = \Comp$, so the map $f$ does not preserve transcendence degree.
\end{example}

From here to the end of this section, we study how $\Ascr(B)$ and $\Kscr(B)$ behave under various operations.
The first operation that we consider is localization, and we restrict ourselves to a special type of localization.
For the notation $A_R$, see the end of the Introduction.

\begin{lemma}  \label {9vbfnaow3GgvVpfhbVWYTE}
Let $B$ be a domain of characteristic zero and $R \in \Ascr(B)$.
\begin{enumerata}

\item For each $A \in \Ascr( B )$ satisfying $R \subseteq A$, we have $A_R \in \Ascr( B_R )$.
Moreover, the map
$$
\setspec{ A \in \Ascr( B ) }{ R \subseteq A } \to \Ascr( B_R ), \quad A \mapsto A_R
$$
is injective and preserves transcendence degree.

\item $\ML(B_R) = R_R$

\item $\setspec{ K \in \Kscr( B ) }{ R \subseteq K } \subseteq \Kscr( B_R )$

\item If $R_R \in \Kscr(B)$ then $\FML(B_R) = R_R$.

\end{enumerata}
\end{lemma}

\begin{proof}
Straightforward, and probably well known.
\end{proof}

Next, we study how $\Ascr(B)$ and $\Kscr(B)$ behave under an algebraic extension of the base field.
We first recall some well-known facts (\ref{0c9v902j3we0fcwpdjhf2873}, \ref{0cv9n2w0dZw0dI28efydldc0} and \ref{Wjd9f23u8n42ndhje8}).

\begin{lemma} \label {0c9v902j3we0fcwpdjhf2873} 
Let $\bk$ be a field of characteristic zero and $B$ a $\bk$-domain.
The following are equivalent:
\begin{enumerata}

\item $\bk$ is algebraically closed in $\Frac(B)$
\item $K \otimes_\bk B$ is a domain for every extension field $K$ of $\bk$
\item $\ck \otimes_\bk B$ is a domain, where $\ck$ is the algebraic closure of $\bk$.

\end{enumerata}
\end{lemma}

% \begin{proof}
% Let $K$ be an extension of $\bk$. If (a) is true then Cor.\ 2 on page 198 of \cite{Zar-Sam_CommAlg1} implies that $K \otimes_\bk \Frac B$ is a domain;
% since $K \otimes_\bk B$ is a subring of $K \otimes_\bk \Frac B$, (b) holds.
% So (a) implies (b), and it is clear that (b) implies (c). Assume that (c) is true.
% Let $\beta$ be an element of $\Frac B$ that is algebraic over $\bk$, and let $f(x) \in \bk[x]$ be the minimal polynomial of $\beta$ over $\bk$.
% Since $\ck \otimes_\bk \Frac B =\ck \otimes_\bk B \otimes_B \Frac B$ is a localization of the domain $\ck \otimes_\bk B$,
% $\ck \otimes_\bk \Frac B$ is a domain.
% As $\ck \otimes_\bk \bk[\beta]$ is a subring of $\ck \otimes_\bk \Frac B$, it is a domain.
% Since $\ck \otimes_\bk \bk[\beta] = \ck \otimes_\bk \bk[x]/(f(x)) = \ck[x]/(f(x))$, it follows that $f(x)$ has degree 1 and that $\beta \in \bk$,
% showing that (a) is true.
% \end{proof}

\begin{lemma} \label {0cv9n2w0dZw0dI28efydldc0}
Consider a tensor product of rings
$$
\xymatrix@R=12pt{
S \ar[r] &   S \otimes_R T \\
R \ar[u] \ar[r]  &   T \ar[u]
}
$$
where we assume that all homomorphisms are injective.
\begin{enumerata}

\item Suppose that $S$ is a free $R$-module and that there exists a basis $\Eeul$ of $S$ over $R$ such that $1 \in \Eeul$.
Then $S \cap T = R$.

\item If $R,S,T$ and $S \otimes_R T$ are domains, and if
$(s_j)_{j \in J}$ is a family of elements of $S$ which is a transcendence basis of $\Frac S$ over $\Frac R$,
then $(s_j \otimes 1)_{j \in J}$ is a transcendence basis of $\Frac(S \otimes_R T)$ over $\Frac T$.

\item If $R,S,T$ and $S \otimes_R T$ are domains then $\trdeg_T(S \otimes_R T) = \trdeg_R S$.

\end{enumerata}
\end{lemma}

\begin{parag} \label {Wjd9f23u8n42ndhje8}
Let $B$ be an algebra over a field $\bk$ of characteristic zero and let $D \in \lnd(B)$. 
Let $\ck$ be any field extension of $\bk$ and define $\bar B = \ck \otimes_\bk B$. 
Applying the functor $\ck \otimes_\bk (\underline{\ \ }) : \text{\rm $\bk$-Mod} \to \text{\rm $\ck$-Mod}$ to $D$
gives a $\ck$-linear map $\bar D : \bar B \to \bar B$, given by $\bar D( \lambda \otimes b ) = \lambda \otimes D(b)$
for all $\lambda \in \ck$ and $b \in B$. It is easily verified that $\bar D \in \lnd( \bar B )$, so we have a well-defined set map
$D \mapsto \bar D$ from $\lnd(B)$ to $\lnd(\bar B)$.
If $D \in \lnd(B)$ and $A = \ker D$, then $\ker(\bar D) = \ck \otimes_\bk A$ because
$\ck \otimes_\bk (\underline{\ \ })$ is an exact functor.
\end{parag}

\begin{lemma} \label {evdhd838eh65433ru0}  
Let $\bk$ be a field of characteristic zero, $B$ be a $\bk$-domain and 
$\ck$ an algebraic extension of $\bk$ such that $\bar B = \ck \otimes_\bk B$ is a domain.
\begin{enumerata}

\item \label {ckiohJGdJhUt823w9eu92-a}  $\trdeg_\ck(\bar B) = \trdeg_\bk B$

\item \label {ckiohJGdJhUt823w9eu92-b}   If $\ML(B)=\bk$ then $\ML(\bar B) = \ck$, and if $\FML(B)=\bk$ then $\FML(\bar B) = \ck$.

\item \label {ckiohJGdJhUt823w9eu92-c}   Each $D \in \lnd(B)$ has a unique extension $\bar D \in \lnd(\bar B)$.
Every subset $\Delta$ of $\lnd(B)$ determines a subset $\bar\Delta$ of $\lnd(\bar B)$ defined by $\bar \Delta = \setspec{ \bar D }{ D \in \Delta }$.
We have
$$
\ck \otimes_\bk A_\Delta = A_{\bar\Delta} \in \Ascr(\bar B) \quad \text{and} \quad \ck \otimes_\bk K_\Delta = K_{\bar\Delta} \in \Kscr(\bar B)
$$
for every subset $\Delta$ of $\lnd(B)$.

\item \label {ckiohJGdJhUt823w9eu92-d}   The maps
$\Ascr(B) \to \Ascr( \bar B )$ ($A \mapsto \ck \otimes_\bk A$) and $\Kscr(B) \to \Kscr( \bar B )$ ($K \mapsto \ck \otimes_\bk K$) 
are injective and preserve transcendence degree:
\begin{gather*}
\trdeg( B : A ) = \trdeg( \bar B : \ck \otimes_\bk A ) \quad \text{for all $A \in \Ascr(B)$,} \\
\trdeg( \Frac(B) : K ) = \trdeg( \Frac(\bar B) : \ck \otimes_\bk K ) \quad \text{for all $K \in \Kscr(B)$.}
\end{gather*}

\item \label {ckiohJGdJhUt823w9eu92-e}   The following diagram is commutative:
$$
\xymatrix@R=15pt{
\Kscr( \bar B ) \ar[r]^{\bar f}  & \Ascr( \bar B )   \\
\Kscr( B ) \ar[r]_f \ar[u]^{k}  & \Ascr( B ) \ar[u]_{a} 
}
$$
where $f(K) = B \cap K$, $\bar f(L) = \bar B \cap L$, $k(K) = \ck \otimes_\bk K$ and $a(A) = \ck \otimes_\bk A$.

\end{enumerata}
\end{lemma}

\begin{proof}
Given any $A \in \Ascr(B)$ and $K \in \Kscr(B)$, we may consider the commutative diagrams:
\begin{equation} \label {Dg192y65eu8}
\begin{array}{ccc}
\xymatrix@R=15pt{
\ck \ar[r]  & \ck \otimes_\bk A \ar[r] & \ck \otimes_\bk B \ar[r] & \ck \otimes_\bk \Frac B \\
\bk \ar[u] \ar[r]  & A  \ar[u]\ar[r] & B  \ar[u]\ar[r] & \Frac B  \ar[u]
} & \quad &
\xymatrix@R=15pt{
\ck \ar[r] & \ck \otimes_\bk K \ar[r] & \ck \otimes_\bk \Frac B \\
\bk \ar[u] \ar[r] & K  \ar[u]\ar[r] & \Frac B  \ar[u]
} \\
\mbox{\rm (a)} \rule{0mm}{5mm} & & \mbox{\rm (b)}
\end{array}
\end{equation}
Since all $\bk$-modules are flat, all homomorphisms in diagrams (\ref{Dg192y65eu8}a) and (\ref{Dg192y65eu8}b) are injective.
Since $\ck \otimes_\bk \Frac B = (\ck \otimes_\bk B) \otimes_B \Frac B$ is a localization of the domain $\ck \otimes_\bk B$, 
$\ck \otimes_\bk \Frac B$ is a domain and consequently all rings in the above diagrams are domains.
Since $\ck$ is integral over $\bk$, all vertical arrows in  (\ref{Dg192y65eu8}a) and (\ref{Dg192y65eu8}b) are integral homomorphisms and
in particular $\ck \otimes_\bk \Frac B$ and  $\ck \otimes_\bk K$ are fields.
Since $\ck \otimes_\bk \Frac B$ is a localization of $\ck \otimes_\bk B$, we have
\begin{equation*} 
\ck \otimes_\bk \Frac B = \Frac(\ck \otimes_\bk B).
\end{equation*}
Lemma \ref{0cv9n2w0dZw0dI28efydldc0}(c) gives $\trdeg_\ck(\ck \otimes_\bk B) = \trdeg_\bk(B)$, so assertion \eqref{ckiohJGdJhUt823w9eu92-a} is proved.

Each $D \in \lnd(B)$ extends uniquely to three derivations:
$$
\bar D \in \lnd( \ck \otimes_\bk B ), \quad
D' \in \Der_\bk( \Frac B ), \text{\quad and\quad}
\bar D' \in \Der_\ck( \ck \otimes_\bk \Frac B ).
$$
The assignment $D \mapsto (\bar D, D', \bar D')$ is well defined, and we shall use these notations throughout the proof below.
Note that $\Frac(\ker D) = \ker D'$ and $\Frac( \ker \bar D ) = \ker \bar D'$ for all $D \in \lnd(B)$, by Lemma \ref{f2i3p0f234hef}.
We also point out that $\bar D'( a \otimes x ) = a \otimes D'( x )$ for all $a \in \ck$ and $x \in \Frac B$.

Choose any subset $\Delta \subseteq \lnd(B)$.
Define the subset $\bar\Delta$ of $\lnd( \ck \otimes_\bk B )$ as in the statement of assertion \eqref{ckiohJGdJhUt823w9eu92-c}.
We claim that the subring $\ck \otimes_\bk A_\Delta$ of $\ck \otimes_\bk B$  is equal to $A_{\bar\Delta}$ and that 
the subfield $\ck \otimes_\bk K_\Delta$ of $\ck \otimes_\bk \Frac B = \Frac(\ck \otimes_\bk B)$ is equal to $K_{\bar\Delta}$. 
To see this, we choose a basis $( \lambda_i )_{i \in I}$ of $\ck$ over $\bk$. 
Recall that if $R$ is a ring and $\bk \subseteq R \subseteq \Frac B$ then $( \lambda_i )_{i \in I}$ is
a basis of $\ck \otimes_\bk R$ over $R$.

Consider $\beta \in \ck \otimes_\bk B$ and write $\beta = \sum_{i \in I_0} \lambda_i \otimes b_i$ with $I_0$ a finite subset of $I$
and $b_i \in B$ for all $i \in I_0$.
Then  $\beta \in A_{\bar \Delta}$
if and only if for each $D \in \Delta$ we have $0 = \bar D(\beta) = \sum_i \lambda_i \otimes D(b_i)$,
if and only if for each $D \in \Delta$ and $i \in I_0$ we have $D(b_i) = 0$,
if and only if all $b_i$ belong to $A_\Delta$,
if and only if $\beta \in \ck \otimes_\bk A_\Delta$.
This shows that $\ck \otimes_\bk A_\Delta = A_{\bar \Delta}$.  

Consider $\xi \in \ck \otimes_\bk \Frac B$ and write
$\xi = \sum_{i \in I_0} \lambda_i \otimes x_i$ with $I_0$ a finite subset of $I$ and $x_i \in \Frac B$ for all $i\in I_0$.
Then  $\xi \in K_{\bar \Delta}$
if and only if for each $D \in \Delta$ we have $\xi \in \Frac(\ker \bar D) = \ker( \bar D')$,
if and only if for each $D \in \Delta$ we have $0 = \bar D'(\xi) = \sum_i \lambda_i \otimes D'(x_i)$,
if and only if for each $D \in \Delta$ and $i \in I_0$ we have $D'(x_i) = 0$,
if and only if all $x_i$ belong to $\bigcap_{D \in \Delta} \ker(D') = \bigcap_{D \in \Delta} \Frac( \ker D ) = K_\Delta$,
if and only if $\xi \in \ck \otimes_\bk K_\Delta$.
This shows that $\ck\otimes_\bk K_\Delta = K_{\bar \Delta}$.   
This proves \eqref{ckiohJGdJhUt823w9eu92-c}.

In particular the two set maps of part (d) are well defined.
The fact that $\Ascr(B) \to \Ascr(\bar B)$ (resp.\ $\Kscr(B) \to \Kscr( \bar B )$) is injective follows from the fact that  
$(\ck \otimes_\bk A) \cap B = A$ in diagram (\ref{Dg192y65eu8}a)
(resp.\ $(\ck \otimes_\bk K) \cap \Frac B = K$ in diagram (\ref{Dg192y65eu8}b)),
which itself follows from Lemma \ref{0cv9n2w0dZw0dI28efydldc0}(a).
Lemma \ref{0cv9n2w0dZw0dI28efydldc0}(c) gives $\trdeg(B:A) = \trdeg(\ck \otimes_\bk B: \ck \otimes_\bk A)$
and  $\trdeg( \Frac(B) : K ) = \trdeg( \ck \otimes_\bk \Frac(B) : \ck \otimes_\bk K )$,
so  \eqref{ckiohJGdJhUt823w9eu92-d} is proved.

If $\ML(B)=\bk$ then $\bk \in \Ascr(B)$, so 
\eqref{ckiohJGdJhUt823w9eu92-c} implies that $\ck \otimes_\bk \bk \in \Ascr( \bar B )$,
i.e., $\ck \in \Ascr(\bar B)$ and hence $\ML(\bar B) = \ck$.
Similarly, if $\FML(B)=\bk$ then $\bk \in \Kscr(B)$, so $\ck \otimes_\bk \bk \in \Kscr( \bar B )$,
so $\FML( \ck \otimes_\bk B) = \ck$. This proves \eqref{ckiohJGdJhUt823w9eu92-b}.

\eqref{ckiohJGdJhUt823w9eu92-e}
Let $\Delta$ be a subset of $\lnd(B)$. Then
$\bar f ( k ( K_\Delta ) ) = \bar f( \ck \otimes_\bk K_\Delta )  = \bar f( K_{ \bar \Delta } )  = K_{ \bar \Delta } \cap \bar B = A_{ \bar \Delta }$
and
$a( f ( K_\Delta ) ) = a( K_\Delta \cap B ) = a( A_\Delta ) =  \ck \otimes_\bk A_\Delta = A_{ \bar \Delta }$,
so $\bar f ( k ( K_\Delta ) ) = a( f( K_\Delta ) )$.
\end{proof}

Our next goal is to describe the relation between $\Kscr(B)$ and $\Kscr( K[B] )$,
where $K$ is any element of $\Kscr(B)$ and $K[B]$ is the $K$-subalgebra of $\Frac(B)$ generated by $B$.
Here, the reader should keep in mind that replacing $B$ by $K[B]$ is neither a localization nor a tensor product.

\begin{lemma}  \label {9dfh923809fjf}
Let $\bk$ be a field of characteristic zero, $B$ an affine $\bk$-domain, and $K \in \Kscr(B)$.
Consider the subring  $\Beul = K[B]$ of $\Frac B$.
Then $\Beul$ is an affine $K$-domain, $\dim\Beul = \trdeg_K( \Frac B )$, $\FML(\Beul)=K$, and
$\setspec{ L \in \Kscr(B) }{ K \subseteq L } \subseteq \Kscr( \Beul )$.
\end{lemma}

\begin{proof}
It is clear that $\Beul$ is an affine $K$-domain and that, consequently,
$\dim\Beul$ is equal to the transcendence degree of $\Frac(\Beul)=\Frac(B)$ over $K$.

Let $\Delta$ be the largest possible subset of $\lnd(B)$ satisfying $K = K_\Delta$;
namely, 
\begin{equation} \label {mhbciiwsjhq9d8y}
\Delta = \setspec{ D \in \lnd(B) }{ K \subseteq \Frac(\ker D) } .
\end{equation}

Let $D \in \Delta$.
Then $D$ has a unique extension $D'\in \Der( \Frac B )$;
as $\ker D' = \Frac(\ker D) \supseteq K$ by Lemma \ref{f2i3p0f234hef}, we have in fact  $D' \in \Der_K( \Frac B )$.
As $D'(K) = \{0\} \subseteq \Beul$ and $D'(B) = D(B) \subseteq B \subseteq \Beul$, it follows that $D'(\Beul) \subseteq \Beul$. 
Let $D'' : \Beul \to \Beul$ be the restriction of $D'$.
Consider the subring
$$
\Nil( D'' ) = \setspec{ x \in \Beul }{ \exists_{ r > 0 }\ {D''}^r(x)=0 }
$$
of $\Beul$ and note that
$K \subseteq \ker(D'') \subseteq \Nil(D'')$ and $B \subseteq \Nil(D) \subseteq \Nil(D'')$; it follows that $\Nil(D'') = \Beul$
and hence that $ D'' \in \lnd(\Beul)$.
We also have
$\ker(D) \subseteq \ker(D'') \subseteq \ker(D') = \Frac( \ker D )$ by Lemma \ref{f2i3p0f234hef},
so $\Frac(\ker D'') = \Frac(\ker D)$.

We showed that each $D \in \Delta$ extends to some (necessarily unique) $D'' \in \lnd(\Beul)$ satisfying
$\Frac(\ker D'') = \Frac(\ker D)$.

Consider $L \in \Kscr(B)$ such that $K \subseteq L$, and choose
$\Delta_1 \subseteq \lnd(B)$ such that $L=K_{\Delta_1}$.
Because of \eqref{mhbciiwsjhq9d8y}, we have $\Delta_1 \subseteq \Delta$.
Consequently, each $D \in \Delta_1$ has a unique extension 
$D'' \in \lnd(\Beul)$ satisfying $\Frac(\ker D'') = \Frac(\ker D)$. Now
$$
L
= \bigcap_{D \in \Delta_1} \Frac(\ker D)
= \bigcap_{D \in \Delta_1} \Frac(\ker D'')
\in \Kscr( \Beul ) ,
$$
showing that $\setspec{ L \in \Kscr(B) }{ K \subseteq L } \subseteq \Kscr( \Beul )$.
It follows that $K \in \Kscr( \Beul )$ and hence that $K \supseteq \FML(\Beul)$.
As  $\Beul^* \subseteq\FML(\Beul)$, we have $K \subseteq \FML(\Beul)$ and hence $\FML(\Beul)=K$.
\end{proof}

\section*{The LND-rank}

Given a domain $B$ of characteristic zero, we proceed to define an element $\lndrank(B)$ of $\Nat \cup \{ \infty \}$
that we call the \textit{LND-rank\/} of $B$.
Paragraph \ref{9238hf3jwrdm49} and Lemma \ref{dkjf;asldkjfa} are preliminaries to the definition of $\lndrank(B)$.
The reader should keep in mind that all quantities considered below
(namely $\sup S_r$, $\sup S_f$, $\dim_L \Span_L( \lnd B )$, $\lndrank(B)$, $\haut(\Ascr(B))$ and $\haut(\Kscr(B))$)
are regarded as elements of $\Nat \cup \{ \infty \}$.
In other words, all infinite cardinals are denoted $\infty$ and we do not distinguish between them.

\begin{parag}  \label {9238hf3jwrdm49}
Let $B$ be a domain of characteristic zero.
Let $L= \Frac B$ and $K=\FML(B)$,
and recall that $\Der_K(L)$ is a vector space over $L$ of dimension $\trdeg_K(L)$.
Each element of $\lnd(B)$ has a unique extension to an element of $\Der_K(L)$, so we may regard $\lnd(B)$ as a subset of $\Der_K(L)$.
Let $\Span_L( \lnd B )$ denote the subspace of $\Der_K(L)$ spanned (over $L$) by the set $\lnd(B)$.
Let us also consider 
the set $S_r$ of all $n \in \Nat$ satisfying:
\begin{enumerate}
\item[$(*)$]  \it there exist $D_1, \dots, D_n \in \lnd(B)$ and $b_1, \dots, b_n \in B$
such that the $n \times n$ matrix $( D_i(b_j) )$ has nonzero determinant in $B$,
\end{enumerate}
and the set $S_f$ of all $n \in \Nat$ satisfying:
\begin{enumerate}
\item[$(**)$] \it there exist $D_1, \dots, D_n \in \lnd(B)$ and $b_1, \dots, b_n \in \Frac B$
such that the $n \times n$ matrix $( D_i(b_j) )$ has nonzero determinant in $\Frac B$
\end{enumerate}
where, in $(**)$, the same notation is used for the element $D_i$ of $\lnd(B)$ and its unique extension to an element of $\Der_K(L)$.
The subscripts `$r$' and `$f$' in the notations $S_r$ and $S_f$ stand for the words `ring' and `field' respectively.
Then:
\end{parag}

\begin{sublemma}  \label {dkjf;asldkjfa}
$\sup S_r = \sup S_f = \dim_L \Span_L( \lnd B )$.
\end{sublemma}

\begin{proof}
It is clear that $S_r \subseteq S_f$, so $\sup S_r \le \sup S_f$.

Let $n \in S_f$, and let us prove that $n \le \dim_L\Span_L( \lnd B )$. We may assume that $n\ge1$.
Pick $D_1, \dots, D_n \in \lnd(B)$ and $x_1, \dots, x_n \in \Frac B$ such that $\det( D_i x_j ) \neq 0$.
Now consider $a_1, \dots, a_n \in L$ such that $D= \sum_{i=1}^n a_i D_i \in \Der_K(L)$ is the zero derivation; then
$$
\left( \begin{array}{ccc} a_1 & \cdots & a_n \end{array} \right)
\left(
\begin{smallmatrix}
D_1x_1 & \cdots & D_1x_n \\
\vdots && \vdots \\
D_nx_1 & \cdots & D_nx_n
\end{smallmatrix}
\right)
= \left( \begin{array}{ccc} D x_1 & \cdots & D x_n \end{array} \right)
= \left( \begin{array}{ccc} 0 & \cdots & 0 \end{array} \right),
$$
so $(a_1 , \dots , a_n) = ( 0 , \dots , 0 )$,
since $\det( D_i x_j ) \neq 0$.
This shows that $D_1, \dots, D_n$ are linearly independent over $L$,
so $n \le \dim_L \Span_L(\lnd B)$. It follows that $\sup S_f \le \dim_L \Span_L(\lnd B)$.

Suppose that $n \in \Nat$ satisfies $n \le \dim_L\Span_L(\lnd B)$, and let us prove that $n \in S_r$.
We may assume that $n\ge1$.
Pick $D_1, \dots, D_n \in \lnd(B)$ linearly independent over $L$.
For each $x \in B$, let $\delta_x = ( D_1 x, \dots, D_n x )
\in B^n \subseteq L^n$;
then define  
$$
U = \Span_L \setspec{ \delta_x }{ x \in B }
\text{\ \ and\ \ }
U^\perp = \setspec{ v \in L^n }
{ \langle v, \delta_x \rangle = 0 \text{ for all } x \in B },
$$
where for $v = (a_1, \dots, a_n), v' = (a_1', \dots, a_n') \in L^n$
we define $\langle v,v' \rangle = \sum_{i=1}^n a_i a_i'$.
We claim that $U^\perp = \{0\}$. 
Indeed, consider $(a_1, \dots, a_n) \in U^\perp$.
Define $D = \sum_{i=1}^n a_i D_i \in \Der_K(L)$.  Then for each $x \in B$
we have
$ D(x) = \sum_{i=1}^n a_i D_i x = \langle (a_1, \dots, a_n), \delta_x \rangle
=0 $,
so $D|_B = 0$ and hence $D=0$.
Since $D_1, \dots, D_n$ are linearly independent over $L$, we obtain $(a_1, \dots, a_n) = (0, \dots, 0)$.
Thus $U^\perp=\{0\}$ and consequently $U=L^n$.
So we can choose $x_1, \dots, x_n \in B$ such that 
$\delta_{x_1}, \dots, \delta_{x_n}$ is a basis of $L^n$.
Then $\det( D_i(x_j) ) \neq 0$, showing that $n \in S_r$.
It follows that $\dim_L\Span_L(\lnd B) \le \sup S_r$, so the Lemma is proved.
\end{proof}

\begin{definition}
Let $B$ be a domain of characteristic zero.
By the \textit{LND-rank\/} of $B$,
denoted $\lndrank(B)$,
we mean the element $\sup S_r = \sup S_f = \dim_L \Span_L( \lnd B )$ of $\Nat \cup \{\infty\}$
(see Lemma \ref{dkjf;asldkjfa}).
\end{definition}

\begin{proposition} \label {0c9vkj2w9Hjghi93u98eid}
For any domain $B$ of characteristic zero, we have
$$
\haut \big( \Ascr(B) \big) \le \haut \big( \Kscr(B) \big) \le \lndrank(B) \le \trdeg\big( \Frac(B) : \FML(B) \big) .
$$
\end{proposition}

\begin{proof}
We have $\haut \big( \Ascr(B) \big) \le \haut \big( \Kscr(B) \big)$ by Lemma \ref{0c9vni3w4e8vno}.

Suppose that $n \in \Nat$ and $K_0, \dots, K_n \in \Kscr(B)$ are such that $K_0 \subset \cdots \subset K_n$ 
(where `$\subset$' is strict inclusion).
For each $i = 0, \dots, n$, define $\Delta_i = \setspec{ D \in \lnd(B) }{ K_i \subseteq \Frac(\ker D) }$.
Then  $K_{\Delta_i} = K_i$ for all $i$. We also have $\Delta_0 \supset \cdots \supset \Delta_n$, so we may choose
$D_i \in \Delta_{i-1} \setminus \Delta_{i}$ for each $i \in \{1, \dots, n\}$. 
Let $D_i' \in \Der( \Frac B )$ be the unique extension of $D_i$ and note that $\ker D_i' = \Frac( \ker D_i )$ by Lemma \ref{f2i3p0f234hef}.
Since $D_i \notin \Delta_{i}$, we have $K_{i} \nsubseteq \Frac(\ker D_i) = \ker D_i'$ and hence we may choose $b_i \in K_{i}$ such that 
$D_i'(b_i) \neq 0$. For each $j$ such that $0 \le j<i$ we have $D_i \in \Delta_{i-1} \subseteq \Delta_j$ so $D_i'(b_j) = 0$.
This shows that the $n \times n$ matrix $( D_i'(b_j) )$ is upper triangular with nonzero entries on the diagonal,
so $n \in S_f$ and hence $n \le \sup S_f = \lndrank(B)$ (notation as in \ref{9238hf3jwrdm49}). So $\haut\big(\Kscr(B)\big) \le \lndrank(B)$.

As in \ref{9238hf3jwrdm49}, let $L = \Frac B$ and $K = \FML(B)$.
Since $\lndrank(B)$ is the dimension of the subspace $\Span_L( \lnd(B) )$ of the $L$-vector space $\Der_K(L)$,
we have $\lndrank(B) \le  \dim_L \Der_K(L) = \trdeg_K(L) =  \trdeg\big( \Frac(B) : \FML(B) \big)$, as desired.
\end{proof}

\begin{remark}
There exist affine domains $B$ for which $\haut \big( \Ascr(B) \big) < \trdeg\big( \Frac(B) : \FML(B) \big)$,
i.e., at least one of the inequalities of Prop.\ \ref{0c9vkj2w9Hjghi93u98eid} is strict.  See Rem.\ \ref{c09vb3489ertgbp0anjs3d}.
\end{remark}

\section{Applications}
\label {Section:applicationsandquestions}

We apply the theory developed in Sections \ref{Section:whatisneededfornextsection}--\ref{Somepreliminaries} to study domains of characteristic zero.
This section is subdivided into unnumbered subsections, each one beginning with a title.

\medskip

\subsection*{Preliminaries}

\begin{parag}  \label {934of0vfvo5rje4i}
The following are some of the known facts that we use in this section.
\begin{enumerata}

\item  \label {hghwhegwhegwheg} {\it Let $A$ be a domain containing a field $\bk$ and such that $\trdeg_\bk(A)=1$.
If $A$ is contained in some affine $\bk$-domain, then $A$ is finitely generated as a $\bk$-algebra.}

\item  \label {dijfjfksdf;kjsz;dk}
{\it Suppose that  $\bk \subseteq A \subseteq B$, where $\bk$ is a field, $B$ is a normal affine $\bk$-domain,
and $A$ is a factorially closed subring of $B$ such that $\trdeg_\bk(A)\le2$.
Then $A$ is finitely generated as a $\bk$-algebra.}

\item  \label {qhwgqhopcxoxpcop}
{\it Let $A \subset B$ be integral domains, where $B$ is finitely generated as an $A$-algebra.
Suppose that $S^{-1}B = ( S^{-1}A )^{[1]}$ where $S$ is a multiplicative set
of $A$ satisfying the following condition: each element of $S$ is a product
of units of $A$ and of prime elements $p$ of $A$ such that
\begin{enumerate}

\item[(i)] $p$ is a prime element of $B$
\item[(ii)] $A \cap pB = pA$
\item[(iii)] $A/pA$ is algebraically closed in $B/pB$.

\end{enumerate}
Then $B = A^{[1]}$.  }

\item \label {d;kfjawherj;aekj}
{\it Let $\bk$ be a field of characteristic zero and $B$ 
a normal affine $\bk$-domain such that $\ML(B) = \bk$ and $\trdeg_\bk(B) = 2$.
Then $\Frac B = \bk^{(2)}$ and each element of $\Ascr_1(B)$ is a $\kk1$.}

\item \label {NEW-nonexistenceforms}
{\it Let $\bk$ be a field, $R$ a $\bk$-algebra and $n \le 2$ a natural number.
If there exists a separable field extension $K/\bk$ such that $K \otimes_\bk R = K^{[n]}$, then $R = \kk n$.}

\end{enumerata}
\end{parag}

\begin{proof}
Refer to \cite[Lemma 1.39]{Miy:Book94} for \eqref{hghwhegwhegwheg},
to \cite{Kamb:absence} and \cite{Rus:formal} for \eqref{NEW-nonexistenceforms},
and (for instance) to \cite[5.3.6, p.~82]{Kol:thesis} for \eqref{d;kfjawherj;aekj}.
Statement \eqref{qhwgqhopcxoxpcop} can be derived from the proof of Theorem~2.3.1 of \cite{R-S:xyz}. 
For \eqref{dijfjfksdf;kjsz;dk}, consider $K = \Frac A$ and note that \cite{Zariski:H14}
implies that $K \cap B$ is finitely generated as a $\bk$-algebra;
since $A$ is factorially closed in $B$ we have $K \cap B = A$, so \eqref{dijfjfksdf;kjsz;dk} follows.
\end{proof}

The following simple observation is also needed:

\begin{lemma} \label {Adc8f23beiv6583ruhf9837wehn0qZ}
Let $\bk$ be a field of characteristic zero and $B$ an absolutely factorial $\bk$-domain.
Then each element of $\Ascr(B)$ is absolutely factorial.
\end{lemma}

\begin{proof}
Let $\ck$ be the algebraic closure of $\bk$.
Let $A \in \Ascr(B)$.
Both $B$ and  $\bar B = \ck \otimes_\bk B$ are UFDs, and we have 
$A \in \Ascr(B)$ and (by Lemma \ref{evdhd838eh65433ru0}) $\ck \otimes_\bk A \in \Ascr( \bar B )$;
thus, by Lemma \ref{0ckj238hqn239hnfpawrnHzb}, $A$ (resp.\ $\ck \otimes_\bk A$) is a factorially closed subring of $B$ (resp.\ of $\bar B$).
As a factorially closed subring of a UFD is a UFD, it follows that $A$ and  $\ck \otimes_\bk A$ are UFDs.
Consequently, $A$ is absolutely factorial.
\end{proof}

We begin with a straightforward consequence of Thm \ref{kfkjep930120wdkoi9} and Cor.\ \ref{cknvlwr90fweodpx}:

\begin{corollary} \label {c9JkGUHtrSwqd73yrh437}
Let $\bk$ be a field of characteristic zero and $B$ an affine $\bk$-domain satisfying \mbox{$\FML(B) = \bk$.}
\begin{enumerata}

\item $B$ is geometrically unirational over $\bk$.

\item If $\dim B \le 4$ then $B$ is geometrically rational over $\bk$.

\item If $\dim B \le 2$ then  $B$ is rational over $\bk$.

\end{enumerata}
\end{corollary}

\begin{proof}
Let $\ck$ be the algebraic closure of $\bk$ and $\bar B = \ck \otimes_\bk B$.
The condition $\FML(B)=\bk$ implies that $\bk$ is algebraically closed in $\Frac B$ (Lemma \ref{0ckj238hqn239hnfpawrnHzb}),
so $\bar B$ is a domain by Lemma \ref{0c9v902j3we0fcwpdjhf2873}.
By Lemma \ref{evdhd838eh65433ru0}, $\dim\bar B = \dim B$ and $\FML(\bar B) = \ck$.
It follows from Thm \ref{kfkjep930120wdkoi9} that $\bar B$ is unirational over $\ck$, i.e., (a) is true.
Assertion (b) follows by applying  Cor.\ \ref{cknvlwr90fweodpx}(b) to $\bar B$, and (c) is a reiteration of Cor.\ \ref{cknvlwr90fweodpx}(a).
\end{proof}

\medskip

\subsection*{A generalization of Thm \ref{kfkjep930120wdkoi9}}
In this subsection we state a result from \cite{Daigle:TrivialFML} that describes
what becomes of Thm \ref{kfkjep930120wdkoi9} when $\bk$ is not assumed to be algebraically closed.  
We also give some immediate consequences of that result.
We begin by introducing some notations.

\begin{definition}
Let $\bk$ be an arbitrary field and $B$ an  affine $\bk$-domain.
Write $\kappa(\pgoth) = B_\pgoth / \pgoth B_\pgoth$ for each $\pgoth \in \Spec B$
and let $n = \dim B$.
Define
$$
\Xeul_\bk(B)=
\text{ set of all prime ideals $\pgoth$ of $B$ satisfying $\kappa(\pgoth) \otimes_\bk B \subseteq \kappa(\pgoth)^{[n]}$}
$$
where the notation  $\kappa(\pgoth) \otimes_\bk B \subseteq \kappa(\pgoth)^{[n]}$ is an abbreviation for the sentence:
{\it there exists an injective homomorphism of $\kappa(\pgoth)$-algebras from $\kappa(\pgoth) \otimes_\bk B$
to a polynomial ring in $n$ variables over $\kappa(\pgoth)$}.
We say that $\Xeul_\bk(B)$ has nonempty interior if some nonempty open subset of $\Spec B$ is included in $\Xeul_\bk(B)$.
\end{definition}

The following is a consequence of Thm 3.8, Cor.\ 1.13 and Cor.\ 3.10 of \cite{Daigle:TrivialFML}.

\begin{theorem} \label {9i3oerXfvdf93p04efeJ}
Let $\bk$ be a field of characteristic $0$ and $B$ an affine $\bk$-domain.
Let $n = \dim B$.
If $\FML(B) = \bk$ then the following are true.
\begin{enumerata}

\item  $\Xeul_\bk(B)$ has nonempty interior.

\item \label {9283rbf9v8gejs3436yhu} $\Frac(B) \otimes_\bk B \subseteq (\Frac B)^{[ n ]}$

\item The following conditions are equivalent:
\begin{enumerata}
\item  $B \subseteq \bk^{[ n ]}$ 
\item  $\bk$-rational points are dense in $\Spec B$
\item  $B$ is unirational over $\bk$.
\end{enumerata}

\item If $\bk$ is algebraically closed or $n \le 2$ then $B \subseteq \bk^{[ n ]}$.

\end{enumerata}
\end{theorem}

The next result gives information about the field extensions $\Frac(B) / K$ with $K \in \Kscr(B)$.
It is a simple application of Thm \ref{9i3oerXfvdf93p04efeJ} in conjunction with Lemma \ref{9dfh923809fjf}.

\begin{corollary} \label {92ed8hdh192h0}
Let $\bk$ be a field of characteristic zero, $B$ an affine $\bk$-domain and $K \in \Kscr(B)$.
Let $n = \trdeg_K( \Frac B )$ and let $\Beul = K[B]$ be the $K$-subalgebra of $\Frac B$ generated by $B$.
\begin{enumerata}

\item \label {9cujHDj39rf} $\FML(\Beul)=K$ and $\Xeul_K( \Beul )$ has nonempty interior.

\item \label {p09Cvn2309Ed} There exists a finite extension $K'/K$ such that $B \subseteq {K'}^{[n]}$ and $K' \otimes_K \Frac(B) \subseteq {K'}^{(n)}$.
In particular, $\Frac(B)/K$ is geometrically unirational.

\item \label {0bbc89rf6qi2h}   If $\Frac(B) / K$ is unirational then $B \subseteq K^{[n]}$.

\item \label {0wihUfUY755vhB}  If $n \le 4$ then $\Frac(B)/K$ is geometrically rational.

\item \label {i8chF239edfh}  If $n \le 2$ then $\Frac(B) = K^{(n)}$ and $B \subseteq K^{[n]}$.

\end{enumerata}
\end{corollary}

\begin{proof}
By Lemma \ref{9dfh923809fjf}, $\Beul$ is an affine $K$-domain satisfying $\dim\Beul = n$ and  $\FML( \Beul ) = K$.
So Thm~\ref{9i3oerXfvdf93p04efeJ} implies that \eqref{9cujHDj39rf}  is true.

\eqref{p09Cvn2309Ed}  By \eqref{9cujHDj39rf},  there exists a dense open subset $U$ of $\Spec \Beul$ such that 
$$
\kappa(\pgoth) \otimes_K \Beul \subseteq \kappa(\pgoth)^{[n]} \ \ \text{for every $\pgoth \in U$}
\qquad \text{(where $\kappa(\pgoth) = \Beul_\pgoth / \pgoth \Beul_\pgoth$).}
$$
Pick any maximal ideal $\pgoth$ of $\Beul$ such that $\pgoth \in U$, and define $K' = \kappa(\pgoth) = \Beul/\pgoth$.
Then $K'$ is a finite extension of $K$ and $B \subseteq \Beul \subseteq K' \otimes_K \Beul \subseteq {K'}^{[n]}$.
As  $K' \otimes_K \Frac(B) = K' \otimes_K \Frac(\Beul)$ is a localization of $K' \otimes_K \Beul$,
we have  $K' \otimes_K \Frac(B) \subseteq {K'}^{(n)}$.
As this implies that $\Frac(B)/K$ is geometrically unirational, this proves \eqref{p09Cvn2309Ed}.

\eqref{0bbc89rf6qi2h}  By Thm \ref{9i3oerXfvdf93p04efeJ}, $\Frac(\Beul) / K$ is unirational $\Leftrightarrow$ $\Beul \subseteq K^{[n]}$.
As $\Frac(\Beul) = \Frac(B)$ and $B \subseteq \Beul$, assertion \eqref{0bbc89rf6qi2h}  follows.

\eqref{0wihUfUY755vhB}  If $n \le 4$ then Cor.\ \ref{c9JkGUHtrSwqd73yrh437} implies that $\Beul$ is geometrically rational over $K$,
i.e.,  that $\Frac(B)/K$ is geometrically rational.

\eqref{i8chF239edfh}  Suppose that $n \le 2$. Since $\dim\Beul = n$,
Cor.\ \ref{cknvlwr90fweodpx}(a) gives $\Frac(\Beul) = K^{(n)}$ and
Thm \ref{9i3oerXfvdf93p04efeJ}(c) gives $\Beul \subseteq K^{[n]}$.
Since $\Frac(B) = \Frac(\Beul)$ and $B \subseteq \Beul$, \eqref{i8chF239edfh} follows.
\end{proof}

\begin{bigremark}
Let $\bk$ be a field of characteristic zero, $B$ an affine $\bk$-domain and $K \in \Kscr_2(B)$.
Then Cor.\ \ref{92ed8hdh192h0}\eqref{i8chF239edfh}  implies that $\Frac(B) = K^{(2)}$ and $B \subseteq K^{[2]}$.
However, there does not necessarily exist an embedding $B \to K^{[2]}$ which extends to an isomorphism $\Frac(B) \to K^{(2)}$.
For instance, let $B = \Reals[x,y,v] / (xy-v^2-1)$, where $\Reals[x,y,v] = \Reals^{[3]}$.
Then $\FML(B) = \Reals \in \Kscr_2(B)$,  so $\Frac(B) = \Reals^{(2)}$ and $B \subseteq \Reals^{[2]}$.
However, by paragraph 4.1 of \cite{Bhat-Rus:GeomFac}, $B$ cannot be birationally embedded in $\Reals^{[2]}$.
\end{bigremark}

\medskip

\subsection*{Extensions of rings belonging to $\Ascr(B)$}
Assuming that $B$ is normal, we give some results on ring extensions $R \subset A$ such that $R,A \in \Ascr(B)$ and $\trdeg_R(A)=1$.
The main result is:

\begin{theorem} \label {doccvobjmExm48wednce}
Let $\bk$ be a field of characteristic zero and $B$ a normal affine $\bk$-domain.
Consider a chain $A_0 \subset \cdots \subset A_n$ $(n\ge1)$ of elements of $\Ascr(B)$ satisfying 
$\trdeg(A_i : A_{i-1}) = 1$ for all $i \in \{1, \dots, n\}$.  Assume that
\begin{equation}
\tag{$*$} \Frac(A_n) \in \Kscr(B) \quad \text{or} \quad \Frac(A_{n-1}) \in \Kscr(B) .
\end{equation}
Then $\Frac(A_{i-1}) \in \Kscr(B)$ and $A_{i-1} \in \Ascr_1^*( A_i )$  for all $i \in \{1, \dots, n\}$,
and in particular
$$
\Frac A_n = (\Frac A_0)^{(n)}.
$$
\end{theorem}

\begin{remarks}
\begin{enumerate}

\item Assumption $(*)$ is satisfied whenever $\trdeg(B : A_n) \le 2$.
(Indeed, this is clear if  $\trdeg(B : A_n) < 2$, so let us assume that $A_n \in \Ascr_2(B)$.
There exists $A \in \Ascr_1(B)$ such that $A_n \subset A$. Then $\Frac(A) \in \Kscr(B)$,
so the sequence $A_n \subset A$ satisfies $(*)$; applying the Theorem to $A_n \subset A$ shows that $\Frac(A_n) \in \Kscr(B)$.)

\item For each $i$ such that $A_i$ is finitely generated as an $A_{i-1}$-algebra, we get $A_{i-1} \in \Ascr_1(A_i)$ by Lemma \ref{0cvh2039edvjpqw0}.

\end{enumerate}
\end{remarks}

For the proof of the Theorem, we need the following facts:

\begin{sublemma}  \label {lckjvZ0293ee0diCs}
Let $\bk$ be a field of characteristic zero and $B$ a normal affine $\bk$-domain.
For each $R \in \Ascr(B)$, $\Frac(R)$ is algebraically closed in $\Frac(B)$.
\end{sublemma}

\begin{proof}
We have $\ML( B_R ) = R_R$ by  Lemma \ref{9vbfnaow3GgvVpfhbVWYTE} so $R_R$ is factorially closed (hence integrally closed) in $B_R$.
Since $B_R$ is normal,  $R_R$ is algebraically closed in $\Frac(B)$.
\end{proof}

\begin{sublemma}  \label {pc09vb239vg10qefh}
Let $\bk$ be a field of characteristic zero and $B$ a normal affine $\bk$-domain satisfying $\FML(B) = \bk$. 
If $R \in \Ascr(B)$ and $\trdeg_\bk(R)=1$ then $R=\kk1$.
\end{sublemma}

\begin{proof}
Let $\ck$ be the algebraic closure of $\bk$, $\bar R = \ck \otimes_\bk R$ and  $\bar B = \ck \otimes_\bk B$.
Arguing as in the proof of Cor.\ \ref{c9JkGUHtrSwqd73yrh437}, we find that 
$\bar B$ is a domain and that $\FML(\bar B) = \ck$,
so Thm \ref{9i3oerXfvdf93p04efeJ} implies that $\bar B \subseteq \ck^{[n]}$ for some $n$; thus  $\bar R \subseteq \ck^{[n]}$. 
We have $\trdeg_{\ck}( \bar R ) = 1$ by Lemma \ref{0cv9n2w0dZw0dI28efydldc0}(c).
Since $R$ is a factorially closed subring of the normal domain $B$, $R$ is normal;
as $\Char\bk=0$, it follows that $R$ is  geometrically normal and hence that $\bar R$ is normal (cf.\  \cite[Tag 037Y]{stacks-project}).
Then Zaks' Theorem \cite{Zaks} implies that $\bar R = \ck^{[1]}$.
By \ref{934of0vfvo5rje4i}\eqref{NEW-nonexistenceforms}, it follows that $R = \kk1$.
\end{proof}

\begin{sublemma} \label {lcv0239ed0sdqwo9}
Let $\bk$ be a field of characteristic zero and $B$ a normal affine $\bk$-domain.
Consider a ring extension $R \subset A$ where $R,A \in \Ascr(B)$ and $\trdeg_R(A)=1$.  
The following implications are true:
$$
A_A \in \Kscr(B)
\implies
R_R \in \Kscr(B)
\implies
R \in \Ascr_1^*(A).
$$
\end{sublemma}

\begin{proof}
% Let $K = R_R$.
% If $\trdeg_A(B) < 1$ then $A=B$ and the claim follows from \ref{pc9293ed0wdjo03}(iv).
% If $\trdeg_A(B) = 1$
% then $A \in \Ascr_1(B)$, so $A_R \in \Ascr_1( B_R )$ by Lemma \ref{9vbfnaow3GgvVpfhbVWYTE};
% as $B_R$ is a normal affine $K$-domain of dimension $2$ satisfying (by Lemma \ref{9vbfnaow3GgvVpfhbVWYTE}) $\ML(B_R) = K$,
% \ref{934of0vfvo5rje4i}\eqref{d;kfjawherj;aekj} implies that $A_R = K^{[1]}$.
% So $A_R = K^{[1]}$ whenever  $\trdeg_A(B) \le 1$.
% 
Assume that $A_A \in \Kscr(B)$
and let $\Delta(R) = \setspec{ D \in \lnd(B) }{ R \subseteq \ker D }$ and $K = K_{ \Delta(R) } \in \Kscr(B)$.
Then $K \cap B = R$ and $R_R \subseteq K \subseteq A_A$. Since $A \nsubseteq K \cap B$, we have $K \neq A_A$, so $K$ is algebraic over $R_R$. 
As $R_R$ is algebraically closed in $\Frac B$ by Lemma \ref{lckjvZ0293ee0diCs}, we get $R_R = K$. This shows that $A_A \in \Kscr(B)$ implies $R_R \in \Kscr(B)$.

Now suppose that $R_R \in \Kscr(B)$ and write $K = R_R$. 
Since $R$ is algebraically closed in $A$, in order to show that $R \in \Ascr_1^*(A)$ it suffices to show that $A_R = K^{[1]}$.
Lemma \ref{9vbfnaow3GgvVpfhbVWYTE} gives $\FML(B_R) = K$ and $A_R \in \Ascr( B_R )$.
As $B_R$ is normal and $\trdeg_K( A_R ) = 1$, we obtain $A_R = K^{[1]}$ by Lemma \ref{pc09vb239vg10qefh}.
\end{proof}

\begin{proof}[Proof of Thm \ref{doccvobjmExm48wednce}]
The result follows from Lemma \ref{lcv0239ed0sdqwo9} by induction on $n$.
\end{proof}

We derive some consequences of Thm \ref{doccvobjmExm48wednce}.
The first one is particularly satisfactory:

\begin{corollary} \label {ld0b23or0f9vwo4398hbg}
Let $\bk$ be a field of characteristic zero and $B$ a factorial affine $\bk$-domain.
Suppose that $A_0 \subset \cdots \subset A_n$ is a chain of elements of $\Ascr(B)$ satisfying $n\ge1$ and
$\trdeg(A_i : A_{i-1}) = 1$ for all $i \in \{1, \dots, n\}$.
Then $A_{i-1} \in \Ascr_1^*( A_i )$  for all $i \in \{1, \dots, n\}$.
In particular, $\Frac A_n = (\Frac A_0)^{(n)}$.
\end{corollary}

\begin{proof}
We have $\Frac(A_n) \in \Kscr(B)$ by Lemma \ref{9vdyf25fjmg752rd}, so this follows from Thm \ref{doccvobjmExm48wednce}.
\end{proof}

\begin{corollary} \label {iIuIuhv9e2s9gGiu92uw}
Let $B$ be a normal affine domain over a field $\bk$ of characteristic zero.
Suppose that $A,R,R' \in \Ascr(B)$ satisfy
\begin{itemize}

\item $R \neq R'$, $R \cup R' \subseteq A$ and $\trdeg(A:R) = 1 = \trdeg(A:R')$;

\item $\{ R_R, R'_{R'} \} \subseteq \Kscr(B)$ or $A_A \in \Kscr(B)$.

\end{itemize}
Then $\Frac A = (\Frac R)^{(1)}$ and  $\Frac R$ is ruled over $\bk$.
\end{corollary}

\begin{proof}
We have $R, R' \in \Ascr_1^*(A)$ by Lemma \ref{lcv0239ed0sdqwo9} (or by Thm \ref{doccvobjmExm48wednce}), so $| \Ascr_1^*(A) | > 1$;
the desired conclusion follows from Prop.\ \ref{ckjv9hrjdulxjkdxlkbrvn}.
\end{proof}

\medskip

\subsection*{Maximal height}
Thm \ref{doccvobjmExm48wednce} can be used to study the situation where the height of $\Ascr(B)$ is maximal, i.e., $\haut\Ascr(B) = \dim B$.
The main result is:

\begin{theorem} \label {d0qk3bfvo09e8u37}
Let $\bk$ be a field of characteristic zero, $B$ an affine $\bk$-domain and
$\bk'$ the algebraic closure of $\bk$ in $\Frac(B)$.
Then $\bk'$ is a finite extension of $\bk$ and the following hold.
\begin{enumerata}

\item $\haut \Ascr(B) \le \haut \Kscr(B) \le n$, where $n= \dim B$.

\item If $\haut \Kscr(B) = n$ then $\FML(B) = \bk'$.

\item If $\haut \Ascr(B) = n$ then $\FML(B) = \bk'$, $\Frac(B) = {\bk'}^{(n)}$ and $B \subseteq {\bk'}^{[n]}$.

\end{enumerata}
\end{theorem}

\begin{proof}% [Proof of Thm \ref{d0qk3bfvo09e8u37}]
It is clear that $\bk'/\bk$ is finite.
We have $\haut \Ascr(B) \le \haut \Kscr(B) \le n$ by Prop.\ \ref{0c9vkj2w9Hjghi93u98eid}, so (a) is true.
Prop.\ \ref{0c9vkj2w9Hjghi93u98eid} also implies that $\haut \Kscr(B) \le \trdeg( \Frac B : \FML B ) \le n$,
so if $\haut\Kscr(B) = n$ then $\FML(B)$ is an algebraic extension of $\bk$, so $\FML(B) = \bk'$ since $\FML(B)$ is algebraically closed in $\Frac(B)$.
So (b) is true.

To prove (c), assume that $\haut \Ascr(B) = n$. Note that $\FML(B) = \bk'$ by (b), so what has to be shown is:
\begin{equation} \label {vc98w3dvc9vrsj57ia}
 \Frac(B) = {\bk'}^{(n)} \text{\ \ and\ \ } B \subseteq {\bk'}^{[n]}.
\end{equation}
We first consider the case where $B$ is normal and $\bk$ is algebraically closed in $\Frac B$.
Since $\haut \Ascr(B) = n$, we may consider a chain $A_0 \subset \cdots \subset A_n$ of elements of $\Ascr(B)$.
Then
$$  %\begin{itemize}
\text{$A_n = B$, $A_0=\bk$  and $\trdeg(A_i : A_{i-1}) = 1$ for all $i \in \{1, \dots, n\}$.}
$$
Note that $\Frac(A_n) \in \Kscr(B)$, so Thm \ref{doccvobjmExm48wednce} implies that $\Frac A_n = (\Frac A_0)^{(n)}$, i.e., $\Frac B = \bk^{(n)}$.
We have already noted that $\FML(B) = \bk' = \bk$; this together with  $\Frac B = \bk^{(n)}$ implies that $B \subseteq \kk n$
by  Thm \ref{9i3oerXfvdf93p04efeJ}.
This proves \eqref{vc98w3dvc9vrsj57ia} in the special case.

For the general case,
consider the normalization $\tilde B$ of $B$ and observe that $\dim \tilde B = \dim B = n$.
Note that $\bk' \subseteq \tilde B$; so $\tilde B$ is a normal affine $\bk'$-domain and $\bk'$ is algebraically closed in $\Frac(\tilde B)$.
We have  $\haut \Ascr(B) \le \haut \Ascr(\tilde B)$ by Lemma \ref{8237ted7f983te} and $\haut \Ascr(\tilde B) \le n$ by  Prop.\ \ref{0c9vkj2w9Hjghi93u98eid},
so $\haut \Ascr(\tilde B) = n$.
By the special case, it follows that 
$\Frac(\tilde B) = {\bk'}^{(n)}$ and $\tilde B \subseteq {\bk'}^{[n]}$;
so $\Frac(B) = {\bk'}^{(n)}$ and $B \subseteq {\bk'}^{[n]}$.
This proves \eqref{vc98w3dvc9vrsj57ia}, so we are done.
\end{proof}

\begin{bigremark}
We saw in Section \ref{SectionSomeClarifications} that the condition $\FML(B)=\bk$ does not imply that $B$ is rational over $\bk$,
even when $\bk$ is algebraically closed.
So it is natural to ask whether one can find a condition {\it on the locally nilpotent derivations of $B$} that implies rationality.
Thm \ref{d0qk3bfvo09e8u37} gives an affirmative answer to this question.
Indeed, 
if we assume that $\bk$ is algebraically closed in $\Frac(B)$ (which is a necessary condition for $B$ to be rational over $\bk$) then the implication
$$
\haut\Ascr(B) = \dim B \ \implies \textit{$B$ is rational over $\bk$ and $B \subseteq \kk n$}
$$
is true by Thm \ref{d0qk3bfvo09e8u37}(c).
\end{bigremark}

\begin{bigremark} \label {c09vb3489ertgbp0anjs3d}
Let $\bk$ be a field of characteristic zero and $B$ an affine $\bk$-domain.
If $\FML(B) = \bk$ and $B$ is not rational over $\bk$ then
Thm \ref{d0qk3bfvo09e8u37} implies that 
$$
\haut \big( \Ascr(B) \big) < \trdeg\big( \Frac(B) : \FML(B) \big),
$$
i.e., at least one of the inequalities of Prop.\ \ref{0c9vkj2w9Hjghi93u98eid} is strict.
Such rings $B$ exist:
see the discussion about implication \eqref{8263fnehqwy1jk}, in Sec.\ \ref{SectionSomeClarifications}.
% assuming that $\bk$ is algebraically closed,
% \cite[Thm 2]{Popov_RatFML2013} implies that there exist affine $\bk$-domains $B$ satisfying $\FML(B)=\bk$,
% $B$ is not stably rational over $\bk$ and $\Spec B$ is a smooth variety.
\end{bigremark}

\medskip

\subsection*{Absolutely factorial domains}
As another application of Thm \ref{doccvobjmExm48wednce} (or more precisely, of Cor.\ \ref{ld0b23or0f9vwo4398hbg}),
we shall now prove the following:

\begin{theorem} \label {Z0cjh93wijbrfpa98gfIa}
Let $B$ be an affine domain over a field $\bk$ of characteristic zero and suppose that $B$ is absolutely factorial.
Let $n = \dim B$.
\begin{enumerata}

\item \label {0c9vnbgtu6u5j}  If $A \in \Ascr_{n-2}(B)$ and $R \in \Ascr_{n-1}(B)$ satisfy $R \subset A$, then $A = R^{[1]}$.

\item \label {udfhr3yhre3jihefob}  Suppose that  $(\ck \otimes_\bk B)^*=\ck^*$, where $\ck$ is the algebraic closure of $\bk$.
Then $R = \bk^{[1]}$ for all $R \in \Ascr_{n-1}(B)$.

\end{enumerata}
\end{theorem}

Some preparation is needed for the proof.

\begin{subdefinition}
Let $B$ be a ring, $D \in \lnd(B)$ and $A = \ker D$.
\begin{enumerata}

\item The set $\pl(D) = D(B) \cap A$ is an ideal of $A$, called the \textit{plinth ideal} of $D$.

\item We say that $D$ is \textit{tight} if $D(B) \subseteq \pl(D) B$.

\end{enumerata}
\end{subdefinition}

\begin{subbigremark} \label {c09fh20Xq0wof8ery293}
Let $B$ be a domain of characteristic zero, $D \in \lnd(B)$ and $A = \ker D$.  If $B = A^{[1]}$, then $D$ is tight.
(Indeed, write $B = A[t]$; then $D = a \frac{d}{dt}$ for some $a \in A$, so $D(B) = aB$ and $\pl(D) = aA$.)
\end{subbigremark}

\begin{sublemma}  \label {0c9f203efd0fiQ23e}
Let $B$ be an algebra over a field $\bk$ of characteristic zero, let $D \in \lnd(B)$ and $A = \ker D$.
Let $\ck$ be any field extension of $\bk$ and define $\bar B = \ck \otimes_\bk B$ and $\bar A = \ck \otimes_\bk A$.
Define $\bar D \in \lnd(\bar B)$ as in \ref{Wjd9f23u8n42ndhje8}.
\begin{enumerata}

\item $\pl(\bar D) = \pl(D) \bar A$

\item Given any ideal $J$ of $B$, we have  $D(B) \subseteq J$ if and only if $\bar D( \bar B ) \subseteq J \bar B$.

\item $D$ is tight if and only if $\bar D$ is tight.

\end{enumerata}
\end{sublemma}

\begin{proof}
(a) Recall from \ref{Wjd9f23u8n42ndhje8} that $\ker(\bar D) = \bar A$.
To prove that $\pl(D) \bar A \subseteq \pl(\bar D)$, we have to show that $1 \otimes a \in \pl( \bar D )$ for all $a \in \pl(D)$.
Let $a \in \pl(D)$; then $a = D(b)$ for some $b \in B$, so $\bar D(1 \otimes b) = 1 \otimes D(b) = 1 \otimes a$, so
$1 \otimes a \in \bar A \cap \bar D( \bar B ) = \pl(\bar D)$.

For the reverse inclusion, consider $\alpha \in \pl(\bar D)$ and let us prove that $\alpha \in \pl(D) \bar A$.
Let $(\lambda_i)_{i \in I}$ be a basis of $\ck$ over $\bk$.
We have $\alpha = \bar D(\beta)$ for some $\beta \in \bar B$.
Write $\beta = \sum_i \lambda_i \otimes b_i$ ($b_i \in B$), then
$0 =  \bar D^2(\beta) =  \sum_i \lambda_i \otimes D^2(b_i)$, so $D^2(b_i) = 0$ for all $i$,
so $D(b_i) \in \pl(D)$ for all $i$.
It follows that $\alpha =  \bar D(\beta) =  \sum_i \lambda_i \otimes a_i$ where $a_i = D(b_i) \in \pl(D)$ for all $i$.
Then  $\alpha =  \sum_i (1 \otimes a_i)(\lambda_i \otimes 1) \in \pl(D)\bar A$.

(b) Let $T : \text{\rm $\bk$-Mod} \to \text{\rm $\ck$-Mod}$ denote the functor $\ck \otimes_\bk (\underline{\ \ })$, and note that $T$
is faithfully exact.
Applying $T$ to  $0 \to J \xrightarrow{\ j\ } B \xrightarrow{\ \pi\ } B/J \to 0$ shows that 
$\ker T(\pi) = \image T(j)$, and since $\image T(j) = J\bar B$, we obtain $\ker T(\pi) = J\bar B$.
On the other hand we have $\bar D( \bar B ) = \image T(D)$, so 
$\bar D( \bar B ) \subseteq J \bar B$
$\Leftrightarrow$
$\image T(D) \subseteq \ker T(\pi)$
$\Leftrightarrow$
$T(\pi \circ D) = T(\pi) \circ T(D) = 0$
$\Leftrightarrow$
$\pi \circ D = 0$,
the last step by faithful exactness of $T$.
So $\bar D( \bar B ) \subseteq J \bar B \Leftrightarrow D(B) \subseteq J$.

(c) Let $J = \pl(D) B$.
By (a) we have $\pl(\bar D) \bar B = \big( \pl(D) \bar A \big) \bar B = \big( \pl(D) B \big) \bar B = J \bar B$,
so
$\bar D$ is tight
$\Leftrightarrow$ $\bar D(\bar B) \subseteq \pl(\bar D)\bar B = J \bar B$
$\Leftrightarrow$ $D(B) \subseteq J$
$\Leftrightarrow$ $D$ is tight,
where we used (b) for the equivalence in the middle.
\end{proof}

\begin{sublemma}  \label {0cf1o2q9wsciq8wgxsa9}
Let $B$ be an affine domain over a field $\bk$ of characteristic zero.
Assume that $\dim B = 2$ and that $B$ is absolutely factorial.
Then the following hold.
\begin{enumerata}

\item \label {pc9vnWN230e9Zckae} $B = A^{[1]}$ for all $A \in \Ascr_1(B)$.

\item \label {p0cvh2p0rfpch4h} If $| \Ascr_1(B) | > 1$ then $B = \kk2$.

\end{enumerata}
\end{sublemma}

\begin{proof}
(a) Let  $A \in \Ascr_1(B)$. As $A$ is absolutely factorial by Lemma \ref{Adc8f23beiv6583ruhf9837wehn0qZ}, it is a UFD.
It is also an affine $\bk$-domain by \ref{934of0vfvo5rje4i}\eqref{hghwhegwhegwheg}, and is one-dimensional,
so $A$ is a PID.

First consider the case where $\bk$ is algebraically closed.
We have $S^{-1}B = ( S^{-1}A )^{[1]}$ with $S = A \setminus \{0\}$.
We know that $A$ is a $\bk$-affine PID, 
so if $p$ is an irreducible element of $A$ then $A/pA = \bk$, so $A/pA$ is algebraically closed in $B/pB$;
since $A$ is factorially closed in $B$, $p$ is prime in $B$ and $A \cap pB = pA$;
thus $B = A^{[1]}$ follows from \ref{934of0vfvo5rje4i}\eqref{qhwgqhopcxoxpcop}.

Now consider the general case.
Choose an irreducible\footnote{A derivation $D: B \to B$ is \textit{irreducible} if $B$ is the only principal ideal $I$ of $B$ such that $D(B) \subseteq I$.}
$D \in \lnd(B)$ such that $\ker D = A$. Since $A$ is a PID,
we have $\pl(D) = aA$ for some $a \in A \setminus \{0\}$.
Let $\ck$ be the algebraic closure of $\bk$ and consider
$\bar B = \ck \otimes_\bk B$, $\bar A = \ck \otimes_\bk A$ and $\bar D \in \lnd(\bar B)$ as in \ref{Wjd9f23u8n42ndhje8}.
By Lemma \ref{evdhd838eh65433ru0}, we have $\bar A = \ker \bar D \in \Ascr_1( \bar B )$ where $\bar B$ is a $\ck$-affine UFD of dimension $2$,
so $\bar B = \bar A ^{[1]}$ by the preceding paragraph. Then $\bar D$ is tight by Rem.\ \ref{c09fh20Xq0wof8ery293},
so Lemma \ref{0c9f203efd0fiQ23e} implies that $D$ is tight, so $D(B) \subseteq aB$.
Since $D$ is irreducible it follows that $a \in B^*$, so $D(B) \cap B^* \neq \emptyset$ and hence $B = A^{[1]}$ by
the Slice Theorem (\ref{pc9293ed0wdjo03}).  This proves (a).

(b) Suppose that $A_1,A_2$ are distinct elements of $\Ascr_1(B)$. Then $B = A_1^{[1]} = A_2^{[1]}$, so by \cite[Thm 3.3]{AEH72}
we have that $A_1$ is a polynomial ring in one variable over the algebraic closure $\bk'$ of $\bk$ in $A_1$.
Since $\ck \otimes_\bk B$ is a domain, it follows from \ref{0c9v902j3we0fcwpdjhf2873} that $\bk$ is algebraically closed in $\Frac B$, so $\bk' = \bk$.
Thus $A_1 = \kk1$ and hence $B = \kk2$.
\end{proof}

\begin{proof}[Proof of Thm \ref{Z0cjh93wijbrfpa98gfIa}]
\eqref{0c9vnbgtu6u5j}  
Suppose that $A \in \Ascr_{n-2}(B)$ and $R \in \Ascr_{n-1}(B)$ satisfy $R \subset A$.
Since $B$ is normal,  $A$ is $\bk$-affine by \ref{934of0vfvo5rje4i}\eqref{dijfjfksdf;kjsz;dk}.
Since $B$ is a UFD, Cor.\ \ref{ld0b23or0f9vwo4398hbg} implies that $R \in \Ascr_1^*(A)$,
so $R \in \Ascr_1(A)$ by Lemma \ref{0cvh2039edvjpqw0}.
By Lemma \ref{Adc8f23beiv6583ruhf9837wehn0qZ}, $A$ is absolutely factorial.
So  Lemma \ref{0cf1o2q9wsciq8wgxsa9}\eqref{pc9vnWN230e9Zckae}  gives $A = R^{[1]}$.

\eqref{udfhr3yhre3jihefob}  
Consider $R \in \Ascr_{n-1}(B)$.
Then $R$ is $\bk$-affine by \ref{934of0vfvo5rje4i}\eqref{hghwhegwhegwheg} and absolutely factorial by Lemma \ref{Adc8f23beiv6583ruhf9837wehn0qZ}.
Thus $\bar R = \ck \otimes_\bk R$ is a $1$-dimensional $\ck$-affine UFD and hence a localization of $\ck^{[1]}$.
The fact that $\bar B^* = \ck^*$ implies that  $\bar R^* = \ck^*$, so $\bar R = \ck^{[1]}$, so
\ref{934of0vfvo5rje4i}\eqref{NEW-nonexistenceforms} gives $R = \kk1$.
\end{proof}

\medskip

\subsection*{A sample of applications}
The results of the above paragraphs can be applied in a variety of situations.
To demonstrate some of the techniques, we give a sample of three propositions.

\begin{proposition} \label {ojHkjgVryul83gcg7g4}
Let $\bk$ be an algebraically closed field of characteristic zero and $B$ a factorial $\bk$-domain satisfying $\FML(B)=\bk$.
Let $n = \dim B$ and suppose that $n \ge5$.
Then $B$ is rational over $\bk$ in each of the following cases:
\begin{enumerata}

% \item \label {idwjengw37ggtr} There exists a chain $R_0 \subset R_1 \subset \cdots \subset R_{n-2}$ in $\Ascr(B)$ with $\trdeg(R_i : R_{i-1}) = 1$
% for all $i \in \{1, \dots, n-2\}$.
\item \label {c9vh304edj0} There exists a chain $A_{n-2} \subset \cdots \subset A_3 \subset A_2$ with $A_i \in \Ascr_i(B)$ for all $i \in \{2, \dots, n-2\}$.
\item \label {oc9v8mf5evng89r} There exist
distinct $A_{n-3}, A_{n-3}' \in \Ascr_{n-3}(B)$
and a chain $A_{n-4} \subset \cdots \subset A_2 \subset A_1$ satisfying $A_i \in \Ascr_i(B)$ for all $i \in \{1, \dots, n-4\}$
and $A_{n-3} \cup A_{n-3}' \subseteq A_{n-4}$.

% \item \label {8c7v8br5jnhxju4ue} $\dim B = 5$ and there exist $A_1 \in \Ascr_1(B)$ and $A_2 \neq A_2'$ in $\Ascr_2(B)$ satisfying $A_2 \cup A_2' \subseteq A_1$.
% \item \label {d09htmcxfkutoxxmgri} $\dim B = 5$ and there exists a chain  $A_3 \subset A_2$ satisfying $A_i \in \Ascr_i(B)$ for $i=2,3$.
% 
% \item $\dim B = 6$ and there exist $A_2 \in \Ascr_2(B)$ and $A_3 \neq A_3'$ in $\Ascr_3(B)$ satisfying $A_3 \cup A_3' \subseteq A_2$.
% 
% \item \label {oc9v8mf5evng89r} $\dim B = 6$ and there exists a chain  $A_4 \subset A_3 \subset A_2$ satisfying $A_i \in \Ascr_i(B)$ for $i=2,3,4$.

\end{enumerata}
\end{proposition}

\begin{proof}
In each of cases  \eqref{c9vh304edj0} and \eqref{oc9v8mf5evng89r}, it suffices to show that there  exists a field $K$ satisfying
\begin{equation} \label {pcihvb76r732}
\bk \subset K \subset \Frac B \quad \text{and} \quad \Frac B = K^{(n-2)} .
\end{equation}
Indeed, we know from Thm \ref{kfkjep930120wdkoi9} that $B$ is unirational over $\bk$, so if \eqref{pcihvb76r732} is true then $K/\bk$ is unirational,
so $K = \bk^{(2)}$ by Castelnuovo's Theorem (Remark 6.2.1, p.\ 422 of \cite{Hartshorne}).
So it is clear that \eqref{pcihvb76r732} implies that $B$ is rational.
Let us prove \eqref{pcihvb76r732}.

In case \eqref{c9vh304edj0}, there exists $A_1 \in \Ascr_1(B)$ satisfying $A_2 \subset A_1$;
applying Cor.\ \ref{ld0b23or0f9vwo4398hbg} to the chain $A_{n-2} \subset \cdots \subset A_1 \subset B$
shows that $\Frac B = (\Frac A_{n-2})^{(n-2)}$, so  $K = \Frac A_{n-2}$ satisfies  \eqref{pcihvb76r732}.

In case \eqref{oc9v8mf5evng89r}, 
applying Cor.\ \ref{ld0b23or0f9vwo4398hbg} to the chain $A_{n-3} \subset A_{n-4} \subset \cdots \subset A_1 \subset B$
shows that $\Frac B = (\Frac A_{n-3})^{(n-3)}$,
and applying Cor.\ \ref{iIuIuhv9e2s9gGiu92uw} to $A_{n-3} \cup A_{n-3}' \subseteq A_{n-4}$ 
implies that there exists a field $K$ such that $\bk \subset K \subset \Frac A_{n-3}$ and $\Frac A_{n-3} = K^{(1)}$;
then $K$ satisfies \eqref{pcihvb76r732} and we are done.
\end{proof}

\begin{proposition} \label {ugHctgJG77856rugXg}
Let $\bk$ be a field of characteristic zero and $B$ an affine $\bk$-domain.
Suppose that $B$ is absolutely factorial, $\dim B = 3$ and $| \Ascr_1(B) | > 1$.
Then $B$ is geometrically rational over $\bk$.
\end{proposition}

\begin{proof}
We first prove the case where $\bk$ is algebraically closed.
Since $B$ is a UFD, Lemma \ref{9vdyf25fjmg752rd} implies that $\FML(B) = \Frac(A)$ where we define $A = \ML(B)$.
Since  $| \Ascr_1(B) | > 1$, we have $A=\bk$ or $A \in \Ascr_2(B)$.
If $A=\bk$ then $\FML(B) = \Frac(A)=\bk$, so $\Frac(B)=\bk^{(3)}$ follows from Cor.\ \ref{cknvlwr90fweodpx}.
If $A \in \Ascr_2(B)$ then (by \ref{934of0vfvo5rje4i}\eqref{hghwhegwhegwheg} and Lemma \ref{Adc8f23beiv6583ruhf9837wehn0qZ})
$A$ is a $1$-dimensional $\bk$-affine UFD, so $\Spec A$ is a factorial curve and hence is rational.
Thus $\FML(B) = \Frac A = \bk^{(1)}$.
As $\FML(B) \in \Kscr_2(B)$, we have $\Frac(B) = (\FML(B))^{(2)}$ by Cor.\ \ref{92ed8hdh192h0}\eqref{i8chF239edfh}, so $\Frac(B) = \bk^{(3)}$.
The special case is proved.

For the general case, let $\ck$ be the algebraic closure of $\bk$ and $\bar B = \ck \otimes_\bk B$.
Note that $\bar B$ is a UFD and an affine $\ck$-domain, and that  $\dim \bar B=3$ and $| \Ascr_1( \bar B ) | > 1$ by Lemma \ref{evdhd838eh65433ru0}.
Since the Proposition is true when $\bk$ is algebraically closed, $\bar B$ is rational over $\ck$; so $B$ is geometrically rational.
\end{proof}

\begin{bigremark} \label {pd9cb93edpd}
{\it Let $L/K$ be a function field of one variable, where $K$ is a field of characteristic zero and is algebraically closed in $L$.
Let $K'$ be an algebraic extension of $K$.  Then $L' = K' \otimes_K L$ is a field, $L'/K'$ is a function field of one variable,
$K'$ is algebraically closed in $L'$ and $L/K$ and $L'/K'$ have the same genus.}
(Indeed, $L'$ is a domain by Lemma \ref{0c9v902j3we0fcwpdjhf2873}, and since it is integral over $L$,
it must be a field; so $L'/K'$ is  a function field of one variable.
We have $L' = K'L$, so Prop.\ III.6.1 and Thm III.6.3 of \cite{StichtenothBook} give the last two claims.)
\end{bigremark}

\begin{proposition}  \label {fb3edf9023e98fyd}
Let $\bk$ be a field of characteristic zero and $B$ an affine $\bk$-domain satisfying
$\dim B = 3$ and $\ML(B)=\bk$.
Also assume that $\ck \otimes_\bk B$ is a domain which is not rational over $\ck$, where $\ck$ is the algebraic closure of $\bk$.
Then the following hold, where we set $F = \FML(B)$:
\begin{enumerata}

\item $\Frac(B)=F^{(2)}$, $B \subseteq F^{[2]}$, and $F$ is the function field of a curve of positive genus over $\bk$.

\item $T_\Ascr(B) = \{0,1,3\}$, $T_\Kscr(B) = \{0,1,2\}$, and the maps
$\xymatrix@1@C=13pt{(\Kscr(B), \subseteq) \ar @<.4ex>[r] &   (\Ascr(B), \subseteq) \ar @<.4ex>[l]  }$ 
of Lemma \ref{f2i3p0f234hef} are isomorphisms of posets.

\item $B$ is not a UFD.

\end{enumerata}
\end{proposition}

\begin{proof}
Cor.\ \ref{92ed8hdh192h0}\eqref{0wihUfUY755vhB} implies that $\Frac(B)/F$ is geometrically rational,
and the hypothesis that $\ck \otimes_\bk B$ is not rational over $\ck$ implies that $\Frac(B)/\bk$ is not geometrically rational; so $F \neq \bk$.
Since $\ck \otimes_\bk B$ is a domain, $\bk$ is algebraically closed in $\Frac(B)$ by Lemma \ref{0c9v902j3we0fcwpdjhf2873}.
Since $F \neq \bk$, this implies that $\trdeg_\bk(F)\ge1$.
On the other hand, $\ML(B)=\bk$ implies that $\trdeg_F( \Frac B ) > 1$, so $\trdeg_F( \Frac B ) = 2$ and hence $T_\Kscr(B) = \{0,1,2\}$.
Thus assertion (b) follows from Lemma \ref{pc0WVEDFWc9vn20wdfc8wgd9s8}.
Since $T_\Ascr(B) \neq T_\Kscr(B)$, $B$ is not a UFD by Rem.\ \ref{09bn349efvkmtZr8}.
We have $\Frac B = F^{(2)}$ and $B \subseteq F^{[2]}$ by Cor.\ \ref{92ed8hdh192h0}\eqref{i8chF239edfh}.
It is clear that $F/\bk$ is the function field of a curve $C$ over $\bk$, and we note that $\bk$ is algebraically closed in $F$.
If $C$ has genus $0$ then (by Rem.\ \ref{pd9cb93edpd}) $\ck \otimes_\bk F = \ck^{(1)}$,
and since $\Frac B = F^{(2)}$ this implies that $\Frac(\ck \otimes_\bk B) = \ck^{(3)}$, contradicting the hypothesis.
So $C$ has positive genus.
\end{proof}

\section{Some interesting classes of algebras}
\label {Section:Formsofkknandofkk3}

This section is an elaborate remark whose aim is to explain how our results apply to certain interesting classes of algebras.
We define two classes $\Ceuls(\bk) \subset \Ceul(\bk)$ of $\bk$-algebras (for any field $\bk$ of characteristic zero)
and then go on to develop two themes:

\noindent $\bullet$ \textit{The class $\Ceuls(\bk)$ is large enough to contain many interesting algebras.}
Paragraph \ref{p0cnd1230dcb9WINFZI2j8s765cdhgh} recalls the definitions of three classes of algebras
that attract much attention from researchers, and shows that those three classes are included in $\Ceuls(\bk)$.
Lemma \ref{c98vqbeg8vsCb8ecnrfgqw7} shows that $\Ceul(\bk)$ and $\Ceuls(\bk)$ are closed under certain operations, which also supports the
claim that those classes are large.

\noindent $\bullet$ \textit{Some strong results about $\Ascr(B)$ are valid for all members $B$ of $\Ceul(\bk)$ or $\Ceuls(\bk)$.}
The assumptions contained in the definition of $\Ceul(\bk)$ are suitable for applying the results of Section \ref{Section:applicationsandquestions},
and doing so gives Cor.\ \ref{d092n3ro9fqp3r0fck}.
The nonsingularity requirement in the definition of $\Ceuls(\bk)$ allows us to obtain Thm \ref{oicbf2938bvcoq76pjma}.

\begin{definition} \label {9ocb23oe9fsdpo0}
Given a field $\bk$ of characteristic zero, let $\Ceul(\bk)$ be the class of $\bk$-algebras $B$ satisfying
$$
\text{$\bar B$ is an affine $\ck$-domain, is a UFD and satisfies $\bar B^* = \ck^*$}
$$
where $\ck$ denotes the algebraic closure of $\bk$ and $\bar B = \ck \otimes_\bk B$.
Let $\Ceuls(\bk)$ be the class of $\bk$-algebras $B$ that satisfy 
$$
\text{$B$ belongs to $\Ceul(\bk)$ and $\Spec(\bar B)$ is a nonsingular variety over $\ck$.}%
%\,\footnote{By definition, $\Spec(\bar B)$ is a nonsingular variety if and only if all its local rings are regular local rings.}
$$
\end{definition}

\begin{lemma} \label {c98vqbeg8vsCb8ecnrfgqw7}
Let $\bk$ be a field of characteristic zero, $\ck$ its algebraic closure, and $B$ a $\bk$-algebra.
\begin{enumerata}

\item $\kk n \in \Ceuls(\bk)$ for all $n \in \Nat$.

\item $\ck \otimes_\bk B \in \Ceul(\ck) \iff B \in \Ceul(\bk)$ \quad and \quad $\ck \otimes_\bk B \in \Ceuls(\ck) \iff B \in \Ceuls(\bk)$

\item For each $m \in \Nat$,\ \ $B^{[m]} \in \Ceul(\bk) \iff  B \in \Ceul(\bk)$ \quad and \quad $B^{[m]} \in \Ceuls(\bk) \iff  B \in \Ceuls(\bk)$.

\end{enumerata}
\end{lemma}

\begin{proof}
Assertion (a) is trivial.
Since $\Ceul(\bk)$ and $\Ceuls(\bk)$ are defined in terms of the properties of $\bar B = \ck \otimes_\bk B$, assertion (b) is trivial.
For (c), we first note that $\ck \otimes_\bk B^{[m]} = \bar B^{[m]}$; so $B^{[m]} \in \Ceul(\bk)$ is equivalent to
$\bar B^{[m]}$ being a UFD and $( \bar B^{[m]} )^* = \ck^*$,
which is equivalent to $\bar B$ being a UFD and $\bar B^* = \ck^*$,
which is equivalent to $B \in \Ceul(\bk)$.
To prove  $B^{[m]} \in \Ceuls(\bk) \Leftrightarrow B \in \Ceuls(\bk)$,
there only remains to verify that $\Spec( \bar B^{[m]} )$ is nonsingular if and only if $\Spec( \bar B )$ is nonsingular;
this is well known (and can be checked via the Jacobian criterion \cite[p.\ 31]{Hartshorne}).
So (c) is proved.
\end{proof}

%%%%%%%%%%%%%%%%%%%%%%%%%%%%%%%%%%%%%%%%%%%%%%%%%%%%%%%%%%%%%%%%%%%%%%%%%%%%%%%%%%%%%%%%%%%%%%%%

\begin{parag} \label {p0cnd1230dcb9WINFZI2j8s765cdhgh}
Recall the definitions of the following three interesting classes of algebras.
\end{parag}

%%%%%%%%%%%%%%%%%%%%%%%%%%%%%%%%%%%%%%%%%%%%%%%%%%%%%%%%%%%%%%%%%%%%%%%%%%%%%%%%%%%%%%%%%%%%%%%%

\begin{subparag} \label {lcvpq2093cBARopd}
Let $\bk$ be a field of characteristic zero and $n$ a positive integer.
A $\bk$-algebra $B$ is called \textit{a form of $\kk n$} if there exists a field extension $K/\bk$ such that $K \otimes_\bk B = K^{[n]}$
(or equivalently, if $\ck \otimes_\bk B = \ck^{[n]}$ where $\ck$ is the algebraic closure of $\bk$).
It is known (see \ref{934of0vfvo5rje4i}\eqref{NEW-nonexistenceforms}) that the implication
\begin{equation} \label {cjhyfitlVubycmxdlnUn} 
\textit{if $B$ is a form of $\kk n$  then $B = \kk n$}
\end{equation}
is true when $n\le 2$ but it is an open question to determine the truth value of \eqref{cjhyfitlVubycmxdlnUn} when $n \ge 3$.
\end{subparag}

%%%%%%%%%%%%%%%%%%%%%%%%%%%%%%%%%%%%%%%%%%%%%%%%%%%%%%%%%%%%%%%%%%%%%%%%%%%%%%%%%%%%%%%%%%%%%%%%

\begin{subparag} \label {cp9b2p39efbv287fasek}
Let $\bk$ be a field of characteristic zero.
A $\bk$-algebra $B$ is called a \textit{stably polynomial algebra over $\bk$}  if there exist $m,n \in \Nat$ such that $B^{[m]} = \kk{m+n}$.
It is known that the implication
\begin{equation} \label {lpdd9c8vn230wd}
\textit{if $B$ is a stably polynomial algebra over $\bk$ then $B = \kk n$ for some $n \in \Nat$} 
\end{equation}
is true when $\dim B\le 2$ (see \cite{Fujita:ZariskiProblem}, \cite{Rus:AffineRuled} for the case $\dim B = 2$),
but it is an open question to determine the truth value of \eqref{lpdd9c8vn230wd} when $\dim B \ge 3$.
\end{subparag}

%%%%%%%%%%%%%%%%%%%%%%%%%%%%%%%%%%%%%%%%%%%%%%%%%%%%%%%%%%%%%%%%%%%%%%%%%%%%%%%%%%%%%%%%%%%%%%%%

\begin{subparag} \label {vc89fb34gukuy6isbeuyfg}
An \textit{exotic $\Comp^n$} is a nonsingular affine $\Comp$-variety that is 
diffeomorphic to $\Reals^{2n}$ as a real manifold but is not isomorphic to $\Comp^n$ as an algebraic $\Comp$-variety
(refer to \cite{Zaid:Exotic_1995} for background on this topic).
It follows from \cite{Ramanujam} that there are no exotic $\Comp^2$,
but examples are known of exotic $\Comp^n$ for all $n\ge3$.
Let us now adopt the following definition:
an \textit{exotic $\Comp^{[n]}$} is an affine $\Comp$-domain $B$ such that the complex affine variety $X=\Spec B$ is an exotic $\Comp^n$.
That is, $B$ is an exotic $\Comp^{[n]}$ if and only if it is an affine $\Comp$-domain, $B \neq \Comp^{[n]}$, 
and $\Spec B$ is a nonsingular $\Comp$-variety which, when viewed as a real manifold, is diffeomorphic to $\Reals^{2n}$.
\end{subparag}

%%%%%%%%%%%%%%%%%%%%%%%%%%%%%%%%%%%%%%%%%%%%%%%%%%%%%%%%%%%%%%%%%%%%%%%%%%%%%%%%%%%%%%%%%%%%%%%%

\begin{subproposition}  \label{OoO98nf4tgw3TWUZVw02wGgkjsr}
If $\bk$ is a field of characteristic zero then $\Ceuls(\bk)$ contains all stably polynomial algebras over $\bk$ and all forms of $\kk n$ for all $n$.
Moreover, $\Ceuls(\Comp)$ contains all exotic $\Comp^{[n]}$ for all $n$.
\end{subproposition}

\begin{proof}
Lemma \ref{c98vqbeg8vsCb8ecnrfgqw7} implies that $\Ceuls(\bk)$ contains all stably polynomial algebras over $\bk$ and all forms of $\kk n$ for all $n$.
Let $B$ be an exotic $\Comp^{[n]}$.
By definition, $B$ is an affine $\Comp$-domain and $\Spec B$ is a nonsingular $\Comp$-variety.
We thank M.\ Zaidenberg for pointing out to us that $B^* = \Comp^*$ by \cite[Cor.\ (1.20;1)]{Fujita:TopNonComSurf}.
By \cite[Thm 1]{Gurjar_ThesisPublished_1980}, $B$ is a UFD. So $B$ belongs to $\Ceuls(\Comp)$.
\end{proof}

%%%%%%%%%%%%%%%%%%%%%%%%%%%%%%%%%%%%%%%%%%%%%%%%%%%%%%%%%%%%%%%%%%%%%%%%%%%%%%%%%%%%%%%%%%%%%%%%

\begin{lemma} \label {7CxjE2gunQhgfwertk765}
Let $\bk$ be a field of characteristic zero.
If $B \in \Ceul(\bk)$ then $B$ is an affine $\bk$-domain, is absolutely factorial and satisfies $B^* = \bk^*$.
\end{lemma}

\begin{proof}
Let $B \in \Ceul(\bk)$ and let $\ck$ be the algebraic closure of $\bk$.
Since $\bar B = \ck \otimes_\bk B$ is an affine $\ck$-domain, it follows that $B$ is an affine $\bk$-domain (this is left to the reader).
Since $\bar B$ is a noetherian UFD with $\bar B^* = \ck^*$, it follows that that $B$ is a UFD with $B^*=\bk^*$
(this claim seems to belong to folklore; we provide a proof in the Appendix, see Lemma \ref{jcvo3wbdwoe9df0sd0oc}).
\end{proof}

We now apply the results of Sec.\ \ref{Section:applicationsandquestions} to the class $\Ceul(\bk)$.

\begin{corollary}  \label {d092n3ro9fqp3r0fck}
Let $\bk$ be a field of characteristic zero and $B \in \Ceul(\bk)$.  Let $n = \dim B$.
\begin{enumerata}

\item \label {special-udfhr3yhre3jihefob}  $R = \bk^{[1]}$ for all $R \in \Ascr_{n-1}(B)$.

\item \label {special-0c9vnbgtu6u5j}  If $A \in \Ascr_{n-2}(B)$ and $R \in \Ascr_{n-1}(B)$ satisfy $R \subset A$, then $A = R^{[1]} = \kk2$.

\item \label {8jakdfygcjrC79jdw} Let $A_0 \subset \cdots \subset A_m$ $(m\ge1)$ be a chain in $\Ascr(B)$ satisfying 
$\trdeg(A_i : A_{i-1}) = 1$ for all $i \in \{1, \dots, m\}$.
Then $A_{i-1} \in \Ascr_1^*( A_i )$  for all $i \in \{1, \dots, m\}$, and in particular $\Frac A_m = (\Frac A_0)^{(m)}$.

% If $R, A \in \Ascr(B)$ satisfy $R \subset A$ and $\trdeg_R(A)=1$,
% then $A_R = (R_R)^{[1]}$; if moreover $A$ is finitely generated as an $R$-algebra, then $R \in \Ascr_1(A)$.

\item \label {9n230wdCin120}  If $\dim B = 3$ and $\Ascr_2(B) \neq \emptyset$ then $\Frac(B) = \bk^{(3)}$.

\item \label {eicGfKtT7675gJg}  If $\dim B = 3$ and $B$ is not semi-rigid then $B$ is geometrically rational over $\bk$.

\end{enumerata}
\end{corollary}

\begin{proof}
Lemma \ref{7CxjE2gunQhgfwertk765} implies that $B$ is absolutely factorial, and $(\ck \otimes_\bk B)^* = \ck^*$ by definition of $\Ceul(\bk)$;
so \eqref{special-udfhr3yhre3jihefob} and \eqref{special-0c9vnbgtu6u5j}  are immediate consequences of Thm \ref{Z0cjh93wijbrfpa98gfIa}.
Part \eqref{8jakdfygcjrC79jdw}  follows from Cor.\ \ref{ld0b23or0f9vwo4398hbg}.
To prove \eqref{9n230wdCin120}, assume that $\dim B = 3$ and that $R \in \Ascr_2(B)$;
then there exists $A \in \Ascr_1(B)$ such that $R \subset A$;
then $A=\kk2$ by \eqref{special-0c9vnbgtu6u5j}, so $\Frac(B) = (\Frac A)^{(1)} = \bk^{(3)}$.
Part \eqref{eicGfKtT7675gJg} follows from Prop.\ \ref{ugHctgJG77856rugXg}.
\end{proof}

We need to introduce another class of $\bk$-algebras, sometimes known under the name of ``special Danielewski surfaces''
(whence the letter ``$\Dan$'' in the notation).

\begin{notation}
Given a field $\bk$ of characteristic zero, we let $\Dan(\bk)$ denote the class of $\bk$-algebras isomorphic to
$ \bk[X,Y,Z] / (XY - \phi(Z)) $ for some nonconstant polynomial in one variable $\phi(Z) \in \bk[Z] \setminus \bk$,
where  $\bk[X,Y,Z] = \kk3$. Note that $\kk2 \in \Dan(\bk)$.
\end{notation}

The following is in fact a special case of a result of \cite{AStrivML}.
It is a very intriguing fact, and it is interesting to state it here in the context of the class $\Ceuls(\bk)$.

\begin{theorem}  \label {oicbf2938bvcoq76pjma}
Let $\bk$ be a field of characteristic zero and $B \in \Ceuls(\bk)$.

Then $B_R \in \Dan( R_R )$ for all $R \in \Ascr_2(B)$.
\end{theorem}

To prove the Theorem, we need the notion of smoothness.

\begin{subparag} \label {pc0v2n3o0vG23F}
Following \cite[Def.\ 28.D]{Matsumura}, we say that a ring homomorphism $f : R \to S$ is \textit{smooth} (or that $S$ is smooth over $R$)
if $f$ is formally smooth for the discrete topologies on $R$ and $S$. Explicitly, this means that $f$ is smooth if and only if
for every commutative diagram~(\ref{ocvj3o49dfnqDIAG0weh}-i)
\begin{equation} \label {ocvj3o49dfnqDIAG0weh}
\text{(i)} \quad \xymatrix@R=15pt{ C \ar[r]^-q  &  C/N  \\ R \ar[u]^u \ar[r]_f  &  S \ar[u]_v }
\qquad \qquad 
\text{(ii)} \quad \xymatrix@R=15pt{ C \ar[r]^-q  &  C/N  \\ R \ar[u]^u \ar[r]_f  &  S \ar[u]_v \ar @{-->} [ul]_(.4){v'} }  
\end{equation}
where $C$ is a ring, $N$ is an ideal of $C$ satisfying $N^2=0$ and $q$ is the canonical epimorphism of the quotient ring,
there exists at least one ring homomorphism $v' : S \to C$ that makes diagram (\ref{ocvj3o49dfnqDIAG0weh}-ii) commute. 
We stress that our terminology for smoothness agrees with those of \cite{Matsumura} and \cite{AStrivML}.
We need the following properties of smoothness:
\begin{enumerata}

\item \label {civbwqjhjejm6icgf}  {\it Let $\bk$ be a field, let $\bk'$ and $A$ be $\bk$-algebras and let  $A' = \bk' \otimes_\bk A$.
Then $A$ is smooth over $\bk$ if and only if $A'$ is smooth over $\bk'$.}

\item \label {0Mvnbp30vsmnbs67rQmx9ufhc}
{\it Let $\bk$ be an algebraically closed field and $B$ an affine $\bk$-domain.
Then $B$ is smooth over $\bk$ if and only if $\Spec B$ is a nonsingular algebraic variety over $\bk$.}

\end{enumerata}
\end{subparag}

\begin{proof}[Proof of {\rm (a)} and {\rm (b)}]
Assertion \eqref{civbwqjhjejm6icgf}  is a special case of  \cite[28.O]{Matsumura}.
For assertion \eqref{0Mvnbp30vsmnbs67rQmx9ufhc} one needs to show that $\bk\to B$ is smooth if and only if $B_\mgoth$ is a regular local ring for every
maximal ideal $\mgoth$ of $B$; this follows from the three results in \cite{stacks-project}
identified by the tags 00TN, 00TC and 00TS. 
(Caution:  What we call `smooth' here is called `formally smooth' in \cite{stacks-project}.)
\end{proof}

\begin{sublemma} \label {Ao9jcvmtrdhjr3kxkjXerbvc5f}
Given a field $\bk$ of characteristic zero and a $\bk$-algebra $B$, 
$$
\text{$B$ belongs to $\Ceuls(\bk)$ $\iff$ $B$ belongs to $\Ceul(\bk)$ and $B$ is smooth over $\bk$.}
$$
\end{sublemma}

\begin{proof}
Let $\ck$ be the algebraic closure of $\bk$ and $\bar B = \ck \otimes_\bk B$.
By \ref{pc0v2n3o0vG23F}\eqref{0Mvnbp30vsmnbs67rQmx9ufhc}, $B \in \Ceuls(\bk)$ is equivalent to
``$B \in \Ceul(\bk)$ and $\bar B$ is smooth over $\ck$,''
which by \ref{pc0v2n3o0vG23F}\eqref{civbwqjhjejm6icgf} is equivalent to
``$B \in \Ceul(\bk)$ and $B$ is smooth over $\bk$.''
\end{proof}

\begin{proof}[Proof of Thm \ref{oicbf2938bvcoq76pjma}]
Let $B \in \Ceuls(\bk)$.
By Lemmas \ref{7CxjE2gunQhgfwertk765} and \ref{Ao9jcvmtrdhjr3kxkjXerbvc5f},
$B$ is a geometrically integral affine $\bk$-domain, is a UFD, and is smooth over $\bk$; 
so $B$ belongs to the class $\Neul(\bk)$ of $\bk$-domains defined in \cite{AStrivML}.
Thus $\Ceuls(\bk)$ is included in $\Neul(\bk)$.
Now \cite[Thm 4.1]{AStrivML} is the following statement:
\begin{quote}
\it
Suppose that $\Beul$ is a localization of a ring belonging to the class $\Neul(\bk)$.
If $K$ is a field such that $K \subset \Beul$, $\trdeg_K\Beul=2$ and $\ML(\Beul) = K$, then $\Beul \in \Dan( K )$.
\end{quote}
So for any $B \in \Neul(\bk)$ and $R \in \Ascr_2(B)$, we have $B_R \in \Dan( R_R )$
(because $\trdeg( B_R : R_R )=2$ and, by Lemma \ref{9vbfnaow3GgvVpfhbVWYTE}, $\ML( B_R ) = R_R$).
The claim follows.
\end{proof}

%	\begin{bigremark} \label {lckjvofughtw74}
%	If $\bk$ is a field of characteristic zero and $B$ is a form of $\kk3$ then:
%	\begin{enumerata}
%	
%	\item \label {ciovndt4rwu-b} $A = \kk2$ for every $A \in \Ascr_1(B)$;
%	
%	\item \label {ciovndt4rwu-e} if $B$ is not rigid then $\Frac(B) = \bk^{(3)}$.
%	
%	\end{enumerata}
%	We recall the proof of these well-known facts.
%	Let $\ck$ be the algebraic closure of $\bk$ and $\bar B = \ck \otimes_\bk B = \ck^{[3]}$.
%	If $A \in \Ascr_1(B)$ then $\ck \otimes_\bk A \in \Ascr_1( \bar B )$ by (say) Lemma \ref{evdhd838eh65433ru0},
%	so $\ck \otimes_\bk A = \ck^{[2]}$ by \ref{coiSvoq73g4or0fv8}\eqref{miyanishi},
%	so $A = \kk2$ by \ref{934of0vfvo5rje4i}\eqref{NEW-nonexistenceforms}, which  proves \eqref{ciovndt4rwu-b}.
%	As $\Frac(B) = (\Frac A)^{(1)}$ for all $A \in \Ascr_1(B)$, \eqref{ciovndt4rwu-e} follows from \eqref{ciovndt4rwu-b}.
%	\end{bigremark}

We conclude by giving a partial result related to the following open question:
if $\bk$ is a field of characteristic zero and $B$ is a form of $\kk3$ which is not rigid, does it follow that $B = \kk3$?

\begin{proposition} \label {kjcvgdkjsdulqn}
Let $\bk$ be a field of characteristic zero and $B$ a form of $\kk3$.
Suppose that $D_1, D_2 \in \lnd(B) \setminus \{0\}$ satisfy $D_1 \circ D_2 = D_2 \circ D_1$ and $\ker(D_1) \neq \ker(D_2)$.

Then $B = \kk3$ and $\ker(D_1) \cap \ker(D_2) = \bk[v]$ for some variable $v$ of $B$.
\end{proposition}

\begin{subparag}  \label {coiSvoq73g4or0fv8}
The proof of the Proposition uses the following known facts.
\begin{enumerata}

\item \label {p9cvbq390eso}
{\it Let $K/\bk$ be an extension of fields of characteristic zero, let $B$ be a $\bk$-algebra and let $f\in B$.
Suppose that $K \otimes_\bk B = K^{[3]}$ and that $f$ is a variable of $K \otimes_\bk B$.
Then $B = \kk3$ and $f$ is a variable of $B$.} (This is \cite[Prop.\ 2.8]{DK1}.)

\item \label {0iovufaktfudrfo}
{\it Let $\bk$ be a field of characteristic zero and $v \in B = \kk3$.
If $S^{-1}B = \bk(v)^{[2]}$ where $S = \bk[v] \setminus \{0\}$, then $v$ is a variable of $B$.} (This is a special case of \cite[Thm 3]{DK1}.)

\end{enumerata}
\end{subparag}

\begin{proof}[Proof of Prop.\ \ref{kjcvgdkjsdulqn}]
Let $R = \ker(D_1) \cap \ker(D_2)$.
The assumptions on $D_1,D_2$ imply that $D_2$ maps $\ker(D_1)$ into itself and that the
restriction $d_2 \in \lnd( \ker D_1 )$ of $D_2$ is not the zero derivation.
Since $R = \ker d_2$, it follows that $\trdeg(B:R)=2$ and hence that $R \in \Ascr_2(B)$.
So  $R = \kk1$ by Cor.\ \ref{d092n3ro9fqp3r0fck}.
Choose $v \in B$ such that $R = \bk[v]$.

Let $\ck$ denote the algebraic closure of $\bk$, $\bar B = \ck \otimes_\bk B = \ck^{[3]}$
and  $\bar D_1, \bar D_2 \in \lnd(\bar B)$ the extensions of $D_1,D_2$.
Then $\ker(\bar D_1) \cap \ker(\bar D_2) = \ck \otimes_\bk R = \ck[v]$.
It follows that the $\ck(v)$-domain $\Beul = \bar B_{ \ck[v] }$ has a pair of commuting derivations $\delta_1, \delta_2 \in \lnd(\Beul)$
satisfying $\ker\delta_1 \cap\ker\delta_2 = \ck(v)$ ($\delta_i$ is obtained by localizing $\bar D_i$).
As is well known, this implies that $\Beul = \ck(v)^{[2]}$ (see for instance \cite[Prop.\ 3.2]{Maubach_03}).
Then \ref{coiSvoq73g4or0fv8}\eqref{0iovufaktfudrfo} implies that $v$ is a variable of $\bar B$.
It then follows from \ref{coiSvoq73g4or0fv8}\eqref{p9cvbq390eso} that $B = \kk3$ and that $v$ is a variable of $B$.
\end{proof}

\section{Appendix}

We provide a proof for Lemma \ref{jcvo3wbdwoe9df0sd0oc}, which seems to belong to folklore.

\newcommand{\pr}{\operatorname{\text{\rm Pr}}}

\begin{lemma}  \label {cccy5hc7cnc8c4cg}
Let $A$ be an algebra over a field $\bk$, 
let $K/\bk$ be an algebraic Galois extension and write $G=\Gal(K/\bk)$ and $A_K = K \otimes_\bk A$.
For each $\theta \in G$, let $\tilde\theta : A_K \to A_K$ be the  $A$-automorphism of $A_K$ given by
$\tilde\theta (\lambda \otimes a ) =\theta(\lambda) \otimes a $ ($\lambda\in K$, $a\in A$).
If $b$ is an element of $A_K$ satisfying
\begin{equation} \label {cijvno2i3JGFdWKJ732e83}
\forall_{ \theta \in G }\ \exists_{\lambda \in K^*}\  \tilde\theta(b) = \lambda b 
\end{equation}
then there exists $\lambda \in K^*$ such that $\lambda b \in A$.
\end{lemma}

\begin{proof}
We may assume that $b \neq 0$.
Choose a field $E$ satisfying $\bk \subseteq E \subseteq K$ (so $A \subseteq A_E  \subseteq A_{K}$), $b \in A_E$ and $E/\bk$ is finite Galois.
For each $\tau \in \Gal(E/\bk)$, let $\tilde \tau \in \Aut_A( A_E )$ be its extension; we claim that
\begin{equation} \label {dckjvOOieEqowedndcoi}
\forall_{ \tau \in \Gal(E/\bk) }\ \exists_{\lambda \in E^*}\  \tilde\tau(b) = \lambda b .
\end{equation}
Indeed, let $\tau \in \Gal(E/\bk)$. Then $\tau$ extends to $\theta \in G$, which extends to $\tilde\theta \in \Aut_A( A_K )$.
By assumption we have $\tilde\theta(b) = \lambda b$ for some $\lambda \in K^*$.
Since $\tilde\theta(b) = \tilde\tau(b)$, we get $\tilde\tau(b) = \lambda b$.
To prove that $\lambda \in E^*$, it suffices to show that $\omega( \lambda ) = \lambda$ for all $\omega \in \Gal(K/E)$
(because $K/E$ is Galois). Consider $\omega \in \Gal(K/E)$ and the corresponding $\tilde\omega \in \Aut_{A_E}(A_K)$;
then $\tilde\tau(b) = \tilde\omega( \tilde\tau(b) ) = \omega(\lambda) \tilde\omega(b)= \omega(\lambda) b$
because $\tilde\tau(b)$ and $b$ belong to $A_E$ and hence are fixed by $\tilde\omega$.
So $\omega(\lambda) b = \tilde\tau(b) = \lambda b$, where $\omega(\lambda), \lambda \in K^*$ and $b \neq 0$; so $\omega(\lambda) = \lambda$.
This shows that $\lambda \in E^*$ and hence that \eqref{dckjvOOieEqowedndcoi} is true.

So it suffices to prove the special case of the Lemma where $K/\bk$ is a \textit{finite} Galois extension.
Observe that the $\lambda$ in \eqref{cijvno2i3JGFdWKJ732e83} is uniquely determined by $\theta$ (because $b \neq 0$). Thus,
for each $\theta \in G = \Gal(K/\bk)$, there exists a unique $\alpha_\theta \in K^*$ satisfying $\tilde\theta(b) = \alpha_\theta b$.
It follows that
\begin{equation} \label {c89vb3nfvWivfcby3uuawe}
\textit{$\alpha_{ \sigma \circ \tau } = \alpha_{\sigma} \, \sigma( \alpha_{\tau} )$ \quad for all $\sigma,\tau \in G$,}
\end{equation}
i.e.,  that  $\big\{ \alpha_{\theta} \big\}_{ \theta \in G }$ is a $1$-cocycle of $G$ in $K^*$. 
Because $K/\bk$ is finite Galois we have $H^1(G,K^*)=1$ by \cite[Thm~10.1, p.~302]{LangAlgebraText},
so $\big\{ \alpha_{\theta} \big\}_{ \theta \in G }$ is a $1$-coboundary, i.e., 
there exists $\mu \in K^*$ satisfying $\alpha_\theta = \theta(\mu)/\mu$ for all $\theta \in G$.
Then $\tilde\theta( \mu^{-1} b ) = \mu^{-1} b$ for all $\theta \in G$.
As $\setspec{ x \in A_K }{ \forall_{\theta \in G}\,\, \tilde\theta(x)=x } = A$,
it follows that $\mu^{-1} b  \in A$, as desired.
\end{proof}

\begin{lemma} \label {jcvo3wbdwoe9df0sd0oc}
Let $\bk$ be a field and $A$ a $\bk$-algebra.
Suppose that there exists an algebraic Galois extension $K/\bk$ such that $K \otimes_\bk A$ is a noetherian UFD with $(K \otimes_\bk A)^* = K^*$.
Then $A$ is a noetherian UFD with $A^* = \bk^*$.
\end{lemma}

\begin{proof}
Let $A_K = K \otimes_\bk A$.
Since $A_K^* = K^*$ and (by Lemma \ref{0cv9n2w0dZw0dI28efydldc0}(a)) $K \cap A = \bk$, we have $A^* = \bk^*$.
Since $A_K$ is noetherian and faithfully flat over $A$, $A$ is noetherian.  
Let $\pgoth$ be a height $1$ prime ideal of $A$ and let $a_1, \dots, a_n$ be a generating set for $\pgoth$.
Let $b$ be the gcd of $a_1, \dots, a_n$ in $A_K$.
Then $b A_K$ is the least element of the set of principal ideals $J$ of $A_K$ that satisfy $\pgoth A_K \subseteq J$,
and consequently every $A$-automorphism of $A_K$ must map $b A_K$ to itself.
So for each $\theta \in \Gal(K/\bk)$, there exists $\lambda \in A_K^* = K^*$ satisfying  $\tilde\theta(b) = \lambda b$.
By Lemma \ref{cccy5hc7cnc8c4cg}, there exists $\lambda \in K^*$ such that $\lambda b \in A$; so we might as well assume that $b \in A$.
Since $A_K$ is integral over $A$, there exists a height $1$ prime ideal $\qgoth$ of $A_K$ such that $\qgoth \cap A = \pgoth$.
We have $\qgoth = q A_K$ for some prime element $q$ of $A_K$ and $a_1, \dots, a_n \in \qgoth$, so $q \mid a_i$ in $A_K$ (for each $i$);
so $b \in \qgoth$ and hence $\pgoth \subseteq b A_K \cap A \subseteq \qgoth \cap A = \pgoth$, i.e.,  $b A_K \cap A = \pgoth$.
The principal ideal $I = bA$ of $A$ satisfies $I = I A_K \cap A = b A_K \cap A = \pgoth$, so $\pgoth$ is a principal ideal of $A$.
So $A$ is a UFD.
\end{proof}

%%%%%%%%%%%%%%%%%%%%%%%%%%%%%%%%%%%%%%%%%%%%%%%%%%%%%%%%%%%%%%%%%%%%%%%%%%%%%%%%%%%%%%%%%%%%%%%%%%%%%%%%%%%%%%%%%%%%%%%%
%%%%%%%%%%%%%%%%%%%%%%%%%%%%%%%%%%%%%%%%%%%%%%%%%%%%%%%%%%%%%%%%%%%%%%%%%%%%%%%%%%%%%%%%%%%%%%%%%%%%%%%%%%%%%%%%%%%%%%%%
%%%%%%%%%%%%%%%%%%%%%%%%%%%%%%%%%%%%%%%%%%%%%%%%%%%%%%%%%%%%%%%%%%%%%%%%%%%%%%%%%%%%%%%%%%%%%%%%%%%%%%%%%%%%%%%%%%%%%%%%
%%%%%%%%%%%%%%%%%%%%%%%%%%%%%%%%%%%%%%%%%%%%%%%%%%%%%%%%%%%%%%%%%%%%%%%%%%%%%%%%%%%%%%%%%%%%%%%%%%%%%%%%%%%%%%%%%%%%%%%%
%%%%%%%%%%%%%%%%%%%%%%%%%%%%%%%%%%%%%%%%%%%%%%%%%%%%%%%%%%%%%%%%%%%%%%%%%%%%%%%%%%%%%%%%%%%%%%%%%%%%%%%%%%%%%%%%%%%%%%%%

\bibliographystyle{alpha}

\begin{thebibliography}{AFK{\etalchar{+}}13}

\bibitem[AFK{\etalchar{+}}13]{Arz_Flen_Kaliman_Kutz_Zaid:FlexAut}
I.~Arzhantsev, H.~Flenner, S.~Kaliman, F.~Kutzschebauch, and M.~Zaidenberg.
\newblock Flexible varieties and automorphism groups.
\newblock {\em Duke Math. J.}, 162(4):767--823, 2013.

\bibitem[AHE72]{AEH72}
S.~Abhyankar, W.~Heinzer, and P.~Eakin.
\newblock On the uniqueness of the coefficient ring in a polynomial ring.
\newblock {\em J. Algebra}, 23:310--342, 1972.

\bibitem[BR95]{Bhat-Rus:GeomFac}
S.~M. Bhatwadekar and K.~P. Russell.
\newblock A note on geometric factoriality.
\newblock {\em Canad. Math. Bull.}, 38(4):390--395, 1995.

\bibitem[Dai]{Dai:IntroLNDs2010}
D.~Daigle.
\newblock Introduction to locally nilpotent derivations.
\newblock Informal lecture notes prepared in 2010, available at
  \verb!http://aix1.uottawa.ca/~ddaigle!

\bibitem[Dai08]{AStrivML}
D.~Daigle.
\newblock Affine surfaces with trivial makar-limanov invariant.
\newblock {\em J. Algebra}, 319:3100--3111, 2008.

\bibitem[Dai18]{Daigle:TrivialFML}
D.~Daigle.
\newblock Rings with trivial {FML}-invariant.
\newblock Preprint, 2018.

\bibitem[DF94]{Dev-Fins:ruled}
J.~K. Deveney and D.~R. Finston.
\newblock Fields of ${G}_a$ invariants are ruled.
\newblock {\em Canad.\ Math.\ Bull.}, 37:37--41, 1994.

\bibitem[DK09]{DK1}
D.~Daigle and S.~Kaliman.
\newblock A note on locally nilpotent derivations and variables of $k[x,y,z]$.
\newblock {\em Canad. Math. Bull.}, 52:535--543, 2009.

\bibitem[DL16]{Dub_Liendo_2016}
A.~Dubouloz and A.~Liendo.
\newblock Rationally integrable vector fields and rational additive group
  actions.
\newblock {\em Internat. J. Math.}, 27(8), 2016.

\bibitem[Fre17]{Freud:Book-new}
G.~Freudenburg.
\newblock {\em Algebraic theory of locally nilpotent derivations}, volume 136
  of {\em Encyclopaedia of Mathematical Sciences}.
\newblock Springer-Verlag, Berlin, second edition, 2017.
\newblock Invariant Theory and Algebraic Transformation Groups, VII.

\bibitem[Fuj79]{Fujita:ZariskiProblem}
T.~Fujita.
\newblock On {Z}ariski problem.
\newblock {\em Proc.\ Japan Acad.}, 55A:106--110, 1979.

\bibitem[Fuj82]{Fujita:TopNonComSurf}
T.~Fujita.
\newblock On the topology of non-complete algebraic surfaces.
\newblock {\em J. Fac. Sci. Univ. Tokyo}, 29:503--566, 1982.

\bibitem[Gur80]{Gurjar_ThesisPublished_1980}
R.~V. Gurjar.
\newblock Affine varieties dominated by {${\bf C}^{2}$}.
\newblock {\em Comment. Math. Helv.}, 55(3):378--389, 1980.

\bibitem[Har77]{Hartshorne}
R.~Hartshorne.
\newblock {\em Algebraic {G}eometry}, volume~52 of {\em {GTM}}.
\newblock Springer-{V}erlag, 1977.

\bibitem[Kam75]{Kamb:absence}
T.~Kambayashi.
\newblock On the absence of nontrivial separable forms of the affine plane.
\newblock {\em J. Algebra}, 35:449--456, 1975.

\bibitem[Kol10]{Kol:thesis}
R.\ Kolhatkar.
\newblock {\em {H}omogeneous locally nilpotent derivations and affine
  {ML}-surfaces}.
\newblock PhD thesis, University of Ottawa, 2010.

\bibitem[{Lan}93]{LangAlgebraText}
{S}. {Lang}.
\newblock {\em Algebra}.
\newblock Addison-Wesley, third edition, 1993.

\bibitem[Lie10]{Liendo_Rationality2010}
A.~Liendo.
\newblock {$\mathbb{G}_a$}-actions of fiber type on affine
  {$\mathbb{T}$}-varieties.
\newblock {\em J. Algebra}, 324:3653--3665, 2010.

\bibitem[{M}at80]{Matsumura}
{H}. {M}atsumura.
\newblock {\em {C}ommutative {A}lgebra}.
\newblock {M}athematics {L}ecture {N}ote {S}eries. {B}enjamin/{C}ummings,
  second edition, 1980.

\bibitem[Mau03]{Maubach_03}
S.~Maubach.
\newblock The commuting derivations conjecture.
\newblock {\em J. Pure Appl. Algebra}, 179:159--168, 2003.

\bibitem[Miy94]{Miy:Book94}
M.~Miyanishi.
\newblock {\em Algebraic geometry}, volume 136 of {\em Translations of
  Mathematical Monographs}.
\newblock American Mathematical Society, 1994.
\newblock Translated from the 1990 Japanese original by the author.

\bibitem[Ohm89]{OhmSurvey}
J.~Ohm.
\newblock On ruled fields.
\newblock {\em S\'em. Th\'eor. Nombres Bordeaux (2)}, 1(1):27--49, 1989.

\bibitem[Pop11]{Popov_Russellfest}
V.~L. Popov.
\newblock On the {M}akar-{L}imanov, {D}erksen invariants, and finite
  automorphism groups of algebraic varieties.
\newblock In {\em Affine algebraic geometry}, volume~54 of {\em CRM Proc.
  Lecture Notes}, pages 289--311. Amer. Math. Soc., Providence, RI, 2011.

\bibitem[Pop13]{Popov_RatFML2013}
V.~L. Popov.
\newblock Rationality and the {FML} invariant.
\newblock {\em J. Ramanujan Math. Soc.}, 28A:409--415, 2013.

\bibitem[Pop14]{Popov_InfiniteDimAlgGps2014}
V.~L. Popov.
\newblock On infinite dimensional algebraic transformation groups.
\newblock {\em Transform. Groups}, 19:549--568, 2014.

\bibitem[Ram71]{Ramanujam}
C.~P. Ramanujam.
\newblock A topological characterization of the affine plane as an algebraic
  variety.
\newblock {\em Ann.\ of Math.}, 94:69--88, 1971.

\bibitem[RS79]{R-S:xyz}
K.P. Russell and A.~Sathaye.
\newblock On finding and cancelling variables in $\mathbf{k}[x,y,z]$.
\newblock {\em J.\ of Algebra}, 57:151--166, 1979.

\bibitem[Rus81]{Rus:AffineRuled}
K.P. Russell.
\newblock On {A}ffine-{R}uled {R}ational {S}urfaces.
\newblock {\em Math.\ Annalen}, 255:287--302, 1981.

\bibitem[Rus02]{Rus:formal}
K.P. Russell.
\newblock Some formal aspects of the theorems of {M}umford-{R}amanujam.
\newblock In {\em Algebra, Arithmetic and Geometry, Mumbai 2000}, Tata
  Institute of Fundamental Research, pages 557--584. Narosa Publishing, 2002.

\bibitem[{Sta}18]{stacks-project}
The {Stacks Project Authors}.
\newblock \textit{Stacks Project}.
\newblock {http://stacks.math.columbia.edu}, 2018.

\bibitem[Sti93]{StichtenothBook}
H.~Stichtenoth.
\newblock {\em Algebraic function fields and codes}.
\newblock Universitext. Springer-Verlag, Berlin, 1993.

\bibitem[vdE00]{VDE:book}
A.~van~den Essen.
\newblock {\em Polynomial automorphisms}, volume 190 of {\em Progress in
  Mathematics}.
\newblock Birkh\"auser, 2000.

\bibitem[Wri81]{Wright:JacConj}
D.~Wright.
\newblock On the jacobian conjecture.
\newblock {\em Illinois J.\ of Math.}, 25:423--440, 1981.

\bibitem[Zai96]{Zaid:Exotic_1995}
M.~Zaidenberg.
\newblock On exotic algebraic structures on affine spaces.
\newblock In {\em Geometric complex analysis ({H}ayama, 1995)}, pages 691--714.
  World Sci. Publ., River Edge, NJ, 1996.

\bibitem[Zak71]{Zaks}
A.~Zaks.
\newblock Dedekind subrings of $k[x_1,\dots,x_n]$ are rings of polynomials.
\newblock {\em Israel J.\ of Mathematics}, 9:285--289, 1971.

\bibitem[Zar54]{Zariski:H14}
O.~Zariski.
\newblock Interpr\'etations alg\'ebro-g\'eom\'etriques du quatorzi\`eme
  probl\`eme de {H}ilbert.
\newblock {\em {B}ull.\ {S}ci.\ {M}ath.}, 78:155--168, 1954.

\end{thebibliography}
\newcommand{\etalchar}[1]{$^{#1}$}

\end{document}